\documentclass{article}
\usepackage{latexsym,amsfonts,amssymb,epsfig,verbatim,}
\usepackage{amsmath,amsthm,amssymb,latexsym,graphics,textcomp}
\usepackage{graphicx}
\usepackage{tikz}
\usepackage{fullpage}
\usepackage{color}
\usepackage{url}

\newtheorem{theorem}{Theorem}[section]

\newtheorem{lemma}[theorem]{Lemma}

\newtheorem{definition}[theorem]{Definition}

\newtheorem{remark}{Remark}

\newcommand{\MP}{\mathcal P}
\newcommand{\MC}{\mathcal C}

\DeclareMathOperator{\Mod}{Mod}
\DeclareMathOperator{\Homeo}{Homeo}
\DeclareMathOperator{\Aut}{Aut}

\newcommand{\mathsym}[1]{{}}
\newcommand{\unicode}[1]{{}}
\begin{document}

\title{Finite rigid subgraphs of pants graphs}
\author{Jes\'{u}s Hern\'{a}ndez Hern\'{a}ndez\thanks{Partially supported by the UNAM Post-Doctoral Scholarship Program 2017 at the CCM-UNAM, and the CNRS-CONACYT UMI International Laboratory Solomon Lefschetz.},  Christopher J. Leininger\thanks{Partially supported by NSF grant DMS-1811518 and NSF grants DMS 1107452, 1107263, 1107367 ``RNMS: GEometric structures And Representation varieties" (the GEAR Network).},
 and \\Rasimate Maungchang\thanks{Partially supported by a new researcher grants sponsored by Ministry of Science and Technology, National Science and Technology Development Agency, Thailand (No. FDA-CO-2561-8553-TH).}.}
\date{\today}
\maketitle

\begin{abstract}
Let $S_{g,n}$ be an orientable surface of genus $g$ with $n$ punctures. We identify a finite rigid subgraph $X_{g,n}$ of the pants graph $\MP(S_{g,n})$, that is, a subgraph with the property that any simplicial embedding of $X_{g,n}$ into any pants graph $\MP(S_{g',n'})$ is induced by an embedding $S_{g,n}\to S_{g',n'}$.  This extends results of the third author for the case of genus zero surfaces.
\end{abstract}
\section{Introduction}
\label{sec:Introduction}
Let $S=S_{g,n}$ be an orientable surface of genus $g$ with $n$ punctures and let $\Mod^\pm(S)=\pi_0(\Homeo(S))$ be the extended mapping class group. Work of Ivanov~\cite{Ivanov} shows that, for most surfaces, the curve complexes $\MC(S)$ have the property that $\Aut(\MC(S))\cong\Mod^\pm(S)$ (see also Korkmaz~\cite{Korkmaz} and Luo~\cite{Luo}).  Aramayona and Leininger~\cite{AL} extended the results and showed that curve complexes contain {\em finite rigid sets}, which are finite subgraphs with the property that any simplicial embedding is a restriction of an element of $\Mod^\pm(S)$. They also constructed an exhaustion of the curve complex by finite rigid sets~\cite{AL2}; see also Hern{\'a}ndez Hern{\'a}ndez \cite{Hern}.  Existence of and exhaustion by finite rigid sets for non-orientable surfaces have recently been obtained by Ilbira and Korkmaz \cite{IK} and Irmak \cite{Irm}.

Margalit~\cite{Mar} proved a result analogous to Ivanov's for the pants graph, where he showed that $\Aut(\MP(S)) \cong \Mod^\pm(S)$ (with a few exceptions). His result was extended by Aramayona~\cite{Aramayona} who showed that any injective simplicial map $\MP(S) \to \MP(S')$ is {\em induced by an embedding} $f \colon S \to S'$ (see below).  Analogous to the results of Aramayona and Leininger, Maungchang showed that the pants graphs of punctured spheres contain finite rigid sets~\cite{Rasimate1}, and in fact one can exhaust the pants graphs by finite rigid sets~\cite{Rasimate2}.

In this paper we extend the results of \cite{Rasimate1} to (essentially) all finite type surfaces. More precisely, we prove the following theorem.

\begin{theorem}[\textbf{Main theorem}]
Let $S_{g,n}$ be an orientable surface of genus $g$ with $n$ punctures such that $3g-3+n > 1$. There exists a finite subgraph $X_{g,n}\subset\MP(S_{g,n})$ such that for any surface $S_{g',n'}$ and any injective simplicial map
\[
\phi:X_{g,n}\to\MP(S_{g',n'}),
\]
there exists a $\pi_1$-injective embedding $f:S_{g,n}\to S_{g',n'}$ that induces $\phi$.  Moreover, if $(g,n) \not \in \{ (2,0),(1,2)\}$, then $f$ is unique up to isotopy.  If $(g,n) \in \{ (2,0),(1,2)\}$ then $f$ is unique up to isotopy and composition with the hyperelliptic involution.
\end{theorem}

In the above theorem, we say that \textbf{$\phi$ is induced by $f$} if there is a multicurve $Q \subset S_{g',n'}$ such that $f(S_{g,n})$ is a component of $S_{g',n'}-Q$ and for any pants decomposition $u$ of $S_{g,n}$, $f^Q(u) = f(u) \cup Q$ is a pants decomposition of $S_{g',n'}$, and determines a simplicial map
\[f^Q:\MP(S_{g,n})\to\MP(S_{g',n'}),\]
satisfying $f^Q|_{X_{g,n}}=\phi$ 

\bigskip

\noindent
\textbf{Outline of the paper.} Section~\ref{sec:Background} contains basic definitions and necessary properties of pants graphs. We extend Maungchang's result \cite{Rasimate1} for pants graph of punctured sphere in Section~\ref{sec:finite rigidity of punctured sphere}, allowing the target to be the pants graph of an arbitrary surface. We prove the main theorem for $S_{1,2}$ in Section~\ref{sec:S12}. Then we use the result to prove the main theorem for $S_{2,0}$ in Section~\ref{sec:S20}. Finally, we combine the results of punctured spheres and $S_{1,2}$ to prove the general case in Section~\ref{sec:Sgn}.

\bigskip

\noindent
{\bf Acknowledgements.} The authors would like to thank Javier Aramayona for useful conversations and the University of Warwick for its hospitality where this work began.

\section{Background and definitions}
\label{sec:Background}
In this section, we provide the necessary definitions and background   material. For more details, see \cite{Aramayona} and \cite{Mar}. Let $S=S_{g,n}$ be an orientable surface of genus $g$ with $n$ holes; we will allow holes to be either punctures or boundary components, and will pass back and forth between the two whenever it is convenient.  A simple closed curve on $S$ is \textbf{essential} if it does not bound a disk or a one-holed disk on $S$. Throughout this paper, a \textbf{curve} is a homotopy class of essential simple closed curves on S, though we will often confuse the homotopy class with a particular representative whenever it is convenient.

The \textbf{geometric intersection number} of two curves $\alpha$ and $\beta$ on $S$, denoted by $i(\alpha,\beta)$, is the minimum number of transverse intersection points among the simple representatives of $\alpha$ and $\beta$. In this paper, the intersection of any two curves refers to their geometric intersection number. If $i(\alpha,\beta)=0$, then $\alpha$ and $\beta$ are \textbf{disjoint}.

A \textbf{multicurve} $Q$ is (the union of) a set of distinct curves on $S$ with pairwise disjoint representatives. A surface homeomorphic to $S_{0,3}$ is called \textbf{a pair of pants}. Given a multicurve $Q$, the \textbf{nontrivial component(s)} of the complement of the curves in $Q$, denoted $(S-Q)_0$, is the union of the components in the complement of $Q$ not homeomorphic to a pair of pants.

A \textbf{pants decomposition} $P$ of $S$ is a multicurve so that $(S-P)_0 = \emptyset$, that is, the complement of the curves in $P$ is a disjoint union of pairs of pants. For a surface $S = S_{g,n}$, its \textbf{complexity} is $\kappa(S) = 3g+n-3$, which when positive is the number of curves in any pants decomposition. The \textbf{deficiency} of a multicurve $Q$ is the number $\kappa(S)-|Q|$. Note that if the deficiency of $Q$ is $1$, then the nontrivial component $(S-Q)_0$ is either $S_{0,4}$ or $S_{1,2}$.

Two pants decompositions $P$ and $P'$ of $S$ differ by an \textbf{elementary move} if there are curves $\alpha, \alpha'$ on $S$ and a deficiency-$1$ multicurve $Q$ such that $P=\{\alpha\}\cup Q, P'=\{\alpha'\}\cup Q$, and $i(\alpha,\alpha')=1 $ if $(S-Q)_0\cong S_{1,1}$ or $i(\alpha,\alpha')=2 $ if $(S-Q)_0\cong S_{0,4}$. Up to a mapping class on the nontrivial component $(S-Q)_0$, there are two types of elementary moves as shown in Figure~\ref{F:elementary moves}. 
\begin{figure}[ht]
\begin{center}
\includegraphics[height=5cm]{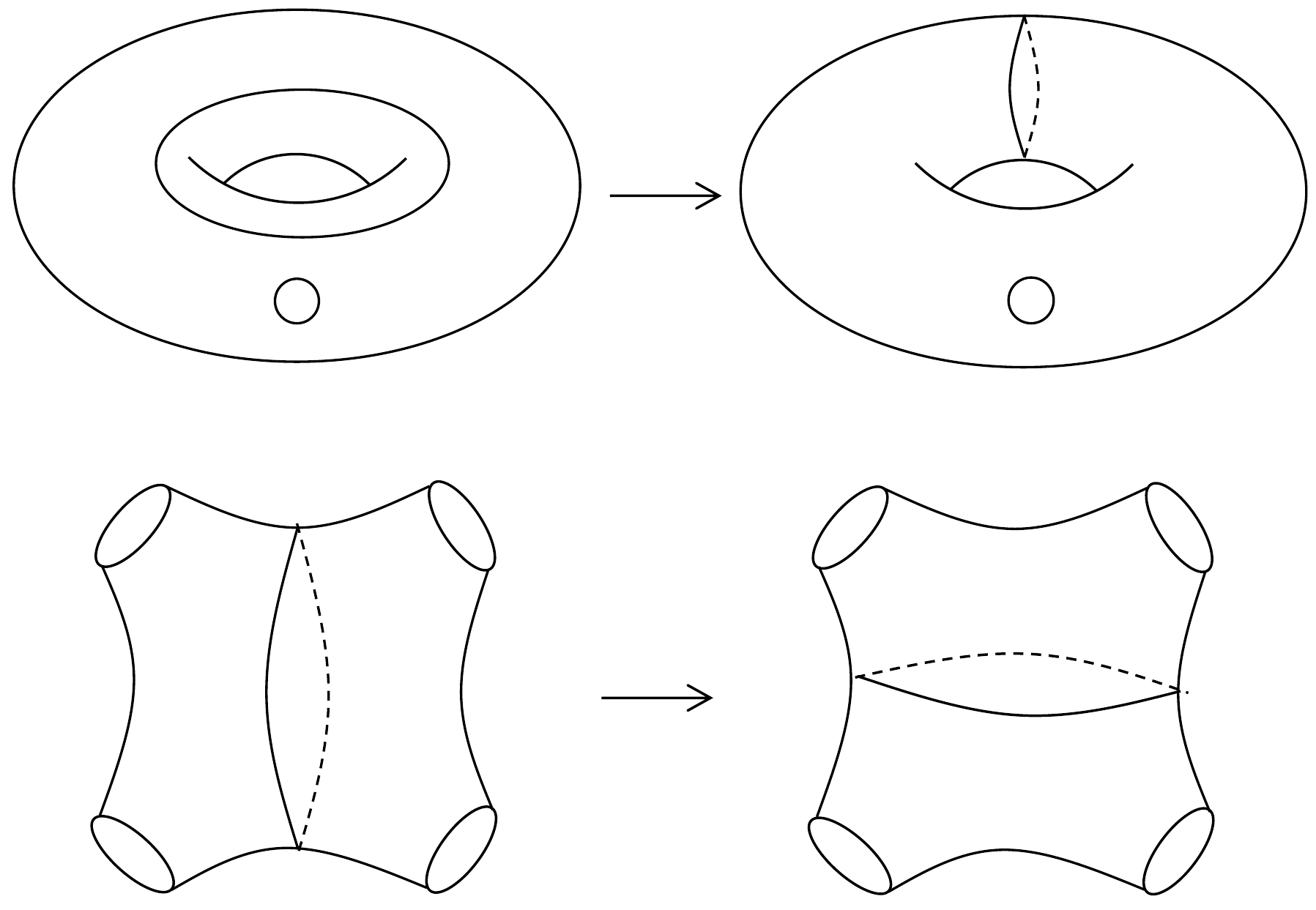} 
\caption{Two types of elementary moves} 
\label{F:elementary moves}
\end{center}
\end{figure}

The \textbf{pants graph} $\MP(S)$ of $S$ is a graph whose vertices correspond to pants decompositions and so that two vertices are connected by an edge if the two corresponding pants decompositions differ by an elementary move. The pants graph is connected, infinite, and locally infinite when $\kappa(S) > 0$; see \cite{HT}. The pants graphs of $S_{1,1}$ and $S_{0,4}$ are isomorphic to a \textbf{Farey graph}, see Figure~\ref{F:Farey graph}. We will not make a distinction between a vertex and its corresponding pants decomposition. Given a multicurve $Q$, we let $\MP_Q(S)$ be the subgraph of $\MP(S)$ {\em spanned} by the set of vertices that contain $Q$ (that is, $\MP_Q(S)$ is the largest subgraph with vertex set consisting of such pants decompositions). Given two adjacent vertices $u$ and $v$ in $\MP(S)$, note that $Q=u\cap v$ is a deficiency-$1$ multicurve and $P_Q(S)$ is a Farey graph. Hence, each edge in $\MP(S)$ is contained in a Farey graph and there are only two triangles attached to each edge.

\begin{figure}[ht]
\begin{center}
\includegraphics[height=7cm]{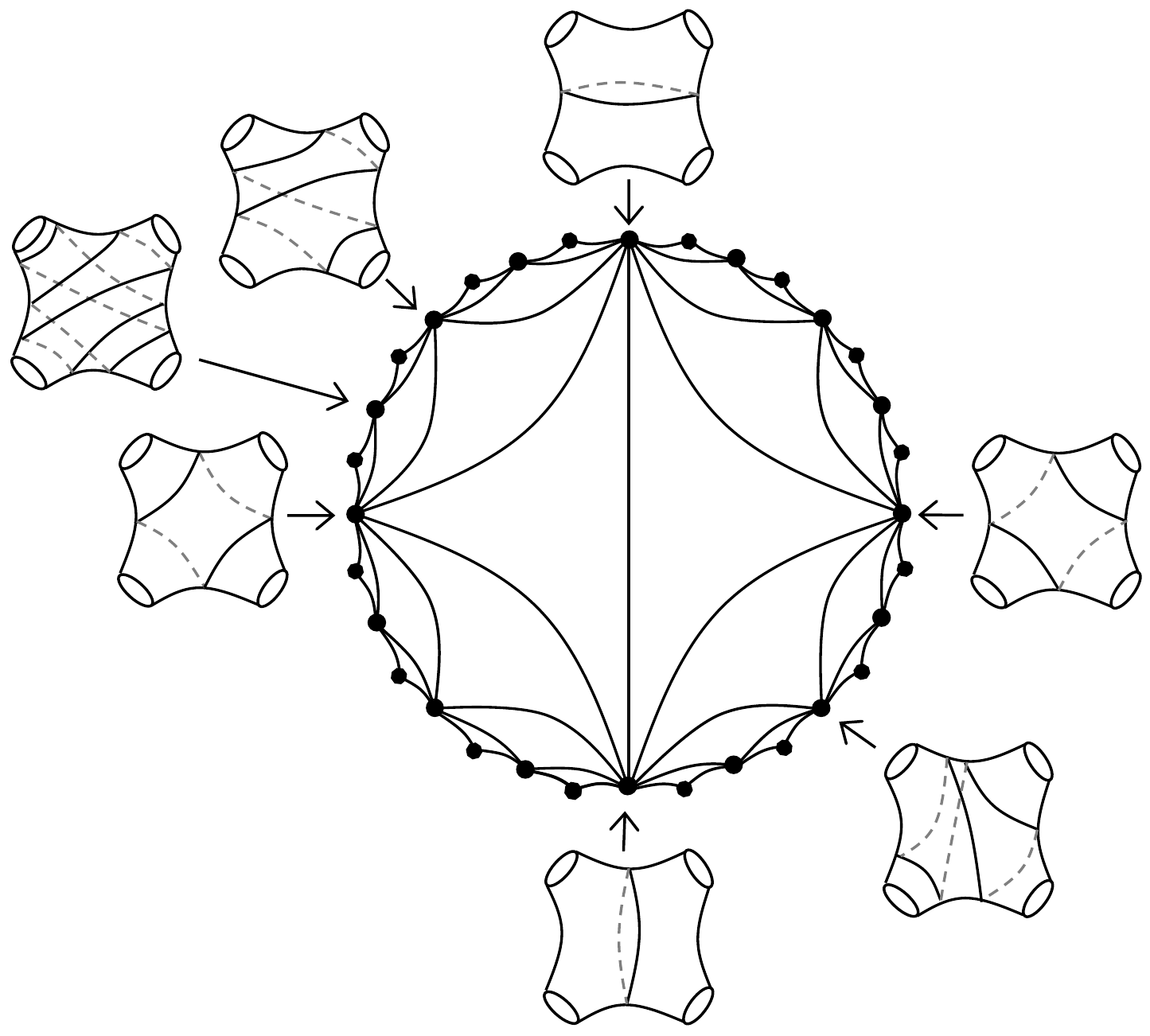} 
\caption{The pants graph $\MP(S_{0,4})$ isomorphic to a Farey graph and some curves representing its vertices.} 
\label{F:Farey graph}
\end{center}
\end{figure}

A \textbf{circuit} in $\MP(S)$ is a subgraph homeomorphic to a circle. A circuit is \textbf{alternating} if any two consecutive edges are in different Farey graphs. We call a circuit a \textbf{triangle}, \textbf{square}, \textbf{pentagon}, or \textbf{hexagon} if it has $3$, $4$, $5$, or $6$ vertices, respectively.

We recall the definition of alternating tuples defined by Aramayona in \cite{Aramayona}, building on Margalit's alternating circuits \cite{Mar}.
\begin{definition}
A cyclically ordered $k$-tuple of distinct vertices $(v_1,v_2,...,v_k)$ in a pants graph is called  an \textbf{alternating $k$-tuple} if $v_i$ and $v_{i+1}$ are in the same Farey graph $F_i$ and $F_i\neq F_{i+1}$ (indices $i$ are taken modulo $k$).
\end{definition}

Alternating $k$-tuples arise naturally from any circuit in $\MP(S_{g,n})$.
\begin{lemma}\label{L:circuits to alternating tuples}  For any circuit of length $m \geq 3$ in $\MP(S_{g,n})$ with vertices $v_1,\ldots,v_m$, either the entire circuit is contained in a single Farey graph, or else there exists $1 \leq i_1 < i_2 < \cdots < i_k \leq m$ so that $v_{i_1},\ldots,v_{i_k}$ is an alternating $k$-tuple.
\end{lemma}
\begin{proof} Each edge of the circuit is contained in a unique Farey graph, and we can decompose the circuit into maximal arcs contained in single Farey graphs.  Either all arcs are in a single Farey graph so the entire circuit is, or endpoints of maximal arcs determine an alternating $k$--tuple.
\end{proof}
\begin{lemma}\label{L:no alternating 2-tuple}
There are no alternating  $2$-tuple and no alternating $3$-tuple.
\end{lemma}
\begin{proof}
Let $v_1$ and $v_2$ be vertices in a Farey graph $F$. Then $F$ is uniquely determined by the deficiency-$1$ multicurve $Q=v_1\cap v_2$, more specifically, $F=\MP_Q(S_{g,n})$. Hence there is no alternating  $2$-tuple.

Suppose $(v_1,v_2,v_3)$ is an alternating $3$-tuple. By definition of alternating tuple, we may write $v_1=$ $\{\alpha_1,$ $\alpha_2,$ $\alpha_3,$ $...,$ $\alpha_{3g+n-3}\},$ $v_2=$ $\{\alpha_1',$ $\alpha_2,$ $\alpha_3,$ $...,$ $\alpha_{3g+n-3}\},$ and $v_3=$ $\{\alpha_1',$ $\alpha_2',$ $\alpha_3,$ $...,$ $\alpha_{3g+n-3}\}$. Observe that $v_1$ and $v_3$ are in the same Farey graph $F=\MP_Q(S_{g,n})$ where $Q=v_1\cap v_3$ is a deficiency-$1$ multicurve but $\alpha_1',\alpha_2'\notin Q$ because $\alpha_1'$ intersects $\alpha_1$ and $\alpha_2'$ intersects $\alpha_2$. This contradicts the fact that $Q$ has deficiency-$1$. So $(v_1,v_2,v_3)$ is not an alternating $3$-tuple. In fact, given a $3$-tuple, if we know that each pair of two consecutive vertices (modulo $3$) is in a Farey graph, then all three vertices must be contained in a Farey graph.  
\end{proof}


The next lemma follows easily from Aramayona \cite[Lemma 8]{Aramayona}.
\begin{lemma}\label{L:square}
Let $(v_1, v_2, v_3, v_4)$ be an alternating 4-tuple. There exists a simplicial bijection $\phi:F_1\to F_3$ such that $\phi(v_1)=v_4$ and $\phi(v_2)=v_3$ . In particular, if there are $m$ paths of length $n$ in $F_1$ connecting $v_1$ and $v_2$, then there are also $m$ paths of length $n$ in $F_3$ connecting $v_3$ and $v_4$.
\end{lemma}
\begin{proof}
Let $(v_1,v_2,v_3,v_4)$ be an alternating $4$-tuple and $F_i$ be a Farey graph containing $v_i, v_{i+1}$ (indices taken modulo $4$).  According to \cite[Lemma 8]{Aramayona}, all $v_i$'s have a deficiency-$2$ multicurve $Q$ in common with two nontrivial components, and we can write $v_1=\{\alpha_1,\alpha_2\}\cup Q$, $v_2=\{\alpha_1',\alpha_2\}\cup Q$, $v_3=\{\alpha_1',\alpha_2'\}\cup Q$, $v_4=\{\alpha_1,\alpha_2'\}\cup Q$, see Figure~\ref{F:square}. Farey graphs $F_1$ and $F_3$ correspond to  deficiency-$1$ multicurves $\{\alpha_2\}\cup Q$ and $\{\alpha_2'\}\cup Q$, respectively.  Moreover, these two multicurves have the same nontrivial component in their complements, which is homeomorphic to either a $4$-punctured sphere or a once-punctured torus. The identity map on the nontrivial component induces a simplicial bijection $\phi: F_1\to F_3$ with the desired properties.
\begin{figure}[ht]
\begin{center}
\includegraphics[height=7cm]{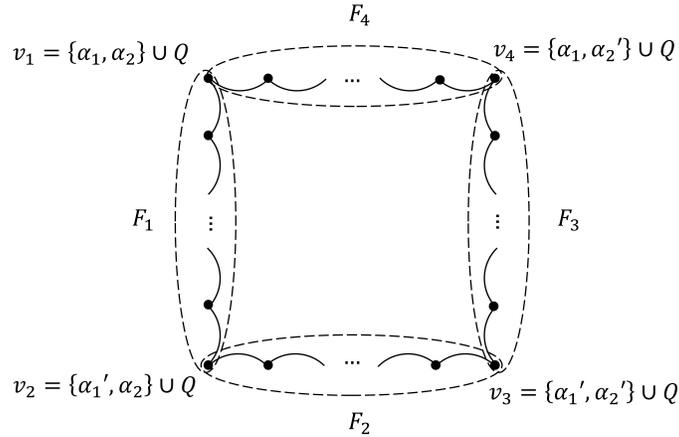} 
\caption{The alternating 4-tuple $(v_1,v_2,v_3,v_4)$.} 
\label{F:square}
\end{center}
\end{figure}
\end{proof}

\begin{lemma}\label{L:length2} A geodesic of length two between points in $\MP(S_{g,n})$ is either the unique geodesic between those points, or else is one of two geodesics between the points.  The latter can only happen when the path is contained in a single Farey graph or is part of an alternating square.
\end{lemma}
\begin{proof} Write $P_0, P_1,P_2$ for the vertices of the path of length $2$.  If all three vertices lie in a single Farey graph, then it is easy to see that there are either one or two geodesics between $P_0$ and $P_2$ in the Farey graph.  Because Farey graphs are totally geodesic \cite{APS}, there can be at most two geodesics between $P_0$ and $P_2$ in this case.

Therefore, we may assume that the path is not contained in a Farey graph, and hence there is a deficiency $2$ multicurve $Q$ and curves $\alpha,\beta,\alpha',\beta'$ so that
\[ P_0 = Q \cup \{\alpha,\beta\}, P_1 = Q \cup \{\alpha',\beta\}, P_2 = Q \cup \{ \alpha',\beta'\}.\]
If $S - Q$ has two components that are complexity $1$ subsurfaces, then the path is part of a square with vertices $P_0,P_1,P_2, P_1' = Q \cup \{ \alpha,\beta'\}$.
 In this case it is clear that these are the only two geodesics between $P_0$ and $P_2$.  

We are thus left to consider the situation that $S - Q$ has one component that is a complexity $2$ subsurface.  
In this case, it suffices to show that $Q \cup \{\alpha,\beta'\}$, $Q \cup \{\alpha,\alpha'\}$, and $Q \cup \{\beta,\beta'\}$ are not pants decompositions.  Since $\alpha$ and $\alpha'$ nontrivially intersect, as do $\beta$ and $\beta'$, the second and third are clearly not pants decompositions. To handle the first case, note that $\alpha'$ is disjoint from both $\beta$ and $\beta'$, but since $\beta'$ intersects the complexity $1$ component of $S - (Q \cup \beta)$, and $\alpha,\alpha'$ fills that component, we must have $\beta'$ intersects $\alpha$.  Therefore, $Q \cup \{\alpha,\beta'\}$ is not a pants decomposition either, completing the proof.
\end{proof}



We recall the construction of the finite rigid subgraph $X_{0,n}\subset\MP{(S_{0,n})}$ for $n\geq5$ in \cite{Rasimate1} as follow.  We view $S_{0,n}$ as being obtained by doubling a regular $n$-gon with all of its vertices removed.  For every arc in the regular $n$--gon connecting non-adjacent sides, the double of this arc is a curve in $S_{0,n}$.  We let $\Gamma_n$ denote the set of all such curves.
Let $Z_{0,n}$ be the subgraph of $\MP{(S_{0,n})}$ induced by the vertices corresponding to pants decompositions consisting of curves from $\Gamma_n$.  See Figure~\ref{F:S05} for the case $n = 5$.

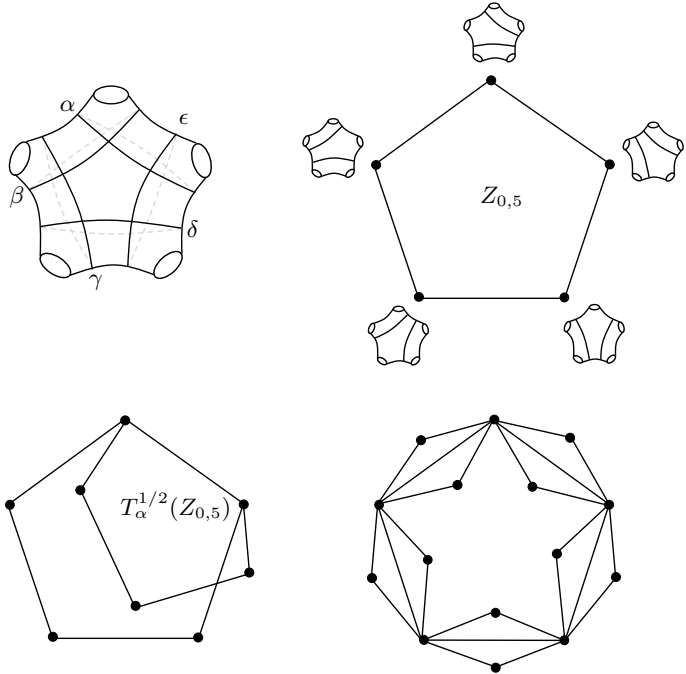
\begin{figure}[ht]
\begin{center}
\includegraphics[height=9cm]{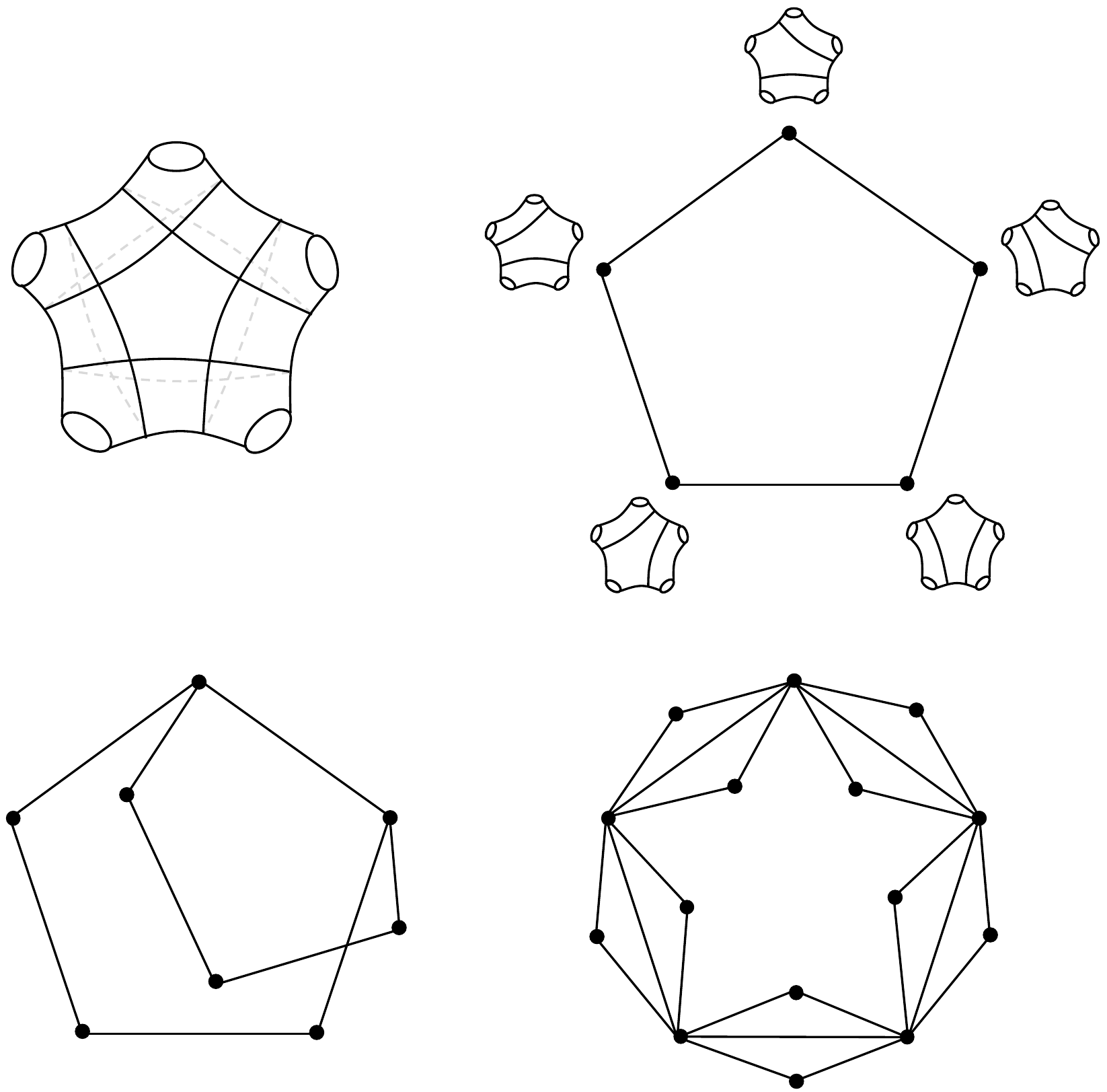} 
\caption{{\bf Top left:} Punctured sphere $S_{0,5}$ and the curves in $\Gamma_5$. {\bf Top right:} The alternating pentagon $Z_{0,5}$.   {\bf Bottom left:} $Z_{0,5} \cup T^{\frac{1}{2}}_\alpha(Z_{0,5})$.  {\bf Bottom right} The thick pentagon $\widehat{Z_{0,5}}$.} 
\label{F:S05}
\end{center}

\begin{picture}(0,0)(0,0)
\put(127,273){\small $\alpha$}
\put(108,238){\small $\beta$}
\put(138,207){\small $\gamma$}
\put(175,225){\small $\delta$}
\put(172,268){\small $\epsilon$}
\put(286,238){\small $Z_{0,5}$}
\put(150,120){\small $T^{1/2}_{\alpha}(Z_{0,5})$}

\end{picture}
\end{figure}

In Figure~\ref{F:S05} we have labeled $\Gamma_{5}=\{\alpha,\beta,\gamma,\delta,\epsilon\}$ and observe that $Z_{0,5}$ is an alternating pentagon. The subgraph
\[
X_{0,5}=Z_{0,5}\cup\bigcup_{c\in\Gamma_5} T_c^{\pm\frac{1}{2}}(Z_{0,5}),
\]
where $T_c^{\pm\frac{1}{2}}$ denotes the two simplicial maps on $\MP(S_{0,5})$ induced by the two half-twists around the curve $c$. The subgraph $X_{0,5}$ consists of the alternating pentagon $Z_{0,5}$ and ten of its images under half-twists. They create ten triangles attached to $Z_{0,5}$ and we call $Z_{0,5}$ together with these triangles, \textbf{thick pentagon} $\widehat{Z_{0,5}}$, see Figure~\ref{F:S05}.  The subgraph $X_{0,5}$ is obtained from $\widehat{Z_{0,5}}$ by adding 10 paths of length 2.  Moreover, if $X'$ is a subgraph of $\MP(S_{0,5})$ isomorphic to $X_{0,5}$ and contains $\widehat{Z_{0,5}}$, then $X' = X_{0,5}$ (see \cite[Lemma 3.3]{Rasimate1}).  In particular, if two vertices in $\widehat{Z_{0,5}}-Z_{0,5}$ are connected by a path of length $2$, that path shows up in $X_{0,5}$ and it is the unique geodesic in the pants graph.

For $n\geq 6$, let $W\subset \Gamma_n$ be a deficiency-$2$ multicurve such that $(S_{0,n}-W)_0\cong S_{0,5}$. Let $\Gamma^W_5$ be a subset of $\Gamma_n$ consisting of curves disjoint from all curves in $W$. There is a natural homeomorphism $h:S_{0,5} \to (S_{0,n}-W)_0$ such that $h(\Gamma_5)=\Gamma_5^W$, see~\cite[Lemma 3.1]{Rasimate1}. Let 
\[ X_{0,5}^W = h^W(X_{0,5}) = \{ h(u) \cup W \mid u \in X_{0,5} \},\]
where $h^W:P(S_{0,5})\to P(S_{0,n})$ is the induced map of $h$ defined by $h^W(u)=h(u)\cup W$. Thus $h^W$ determines an isomorphism $X_{0,5}^W\cong X_{0,5}$.   The set $X_{0,n}$ is then defined to be 
\[X_{0,n} = Z_{0,n} \cup \bigcup_W X^W_{0,5},\]
where the union is taken over all deficiency-$2$ multicurves in $\Gamma_n$ with a $5$-punctured sphere component.

\section{Extending finite rigidity for $\MP(S_{0,n})$}
\label{sec:finite rigidity of punctured sphere}
In~\cite[Theorem 1.1]{Rasimate1}, the third author showed that if $\phi:X_{0,n}\to \MP(S_{0,m})$ is an injective simplicial map, then there is a $\pi_1$-injective embedding $f:S_{0,n}\to S_{0,m}$ that induces $\phi$. In this section, we will show that this theorem can be extended so that $\phi$ is any injective simplicial map from $X_{0,n}$ to an arbitrary pants graph $\MP(S_{g,m})$, as opposed to that of a punctured sphere. We state the theorem as follow.

\begin{theorem}
[Main Theorem for $S_{0,n}$] \label{T:Main for S0n}
For $n\geq5$, there exists a finite subgraph $X_{0,n}\subset\MP(S_{0,n})$ such that for any surface $S_{g,m}$ and any injective simplicial map
\[
\phi:X_{0,n}\to\MP(S_{g,m}),
\]
there exists a deficiency $n-3$ multicurve $Q$ and homeomorphism $f:S_{0,n}\to (S_{g,m}-Q)_0$ so that $f^Q|_{X_{0,n}} = \phi$.  Moreover, $f$ is unique up to isotopy. 
\end{theorem}

The proof of the main theorem in \cite{Rasimate1} is an induction on $n$.  The assumption that the target pants graph was that of a punctured sphere was only used in the base case, $n = 5$; see \cite[Lemma 4.1]{Rasimate1}.  The proof in that case hinges on showing that for any simplicial embedding of $X_{0,5}$ into a pants graph, the image of $Z_{0,5}$ is an alternating pentagon (this is then necessarily determined by a deficiency two multicurve with a single complementary component that is a five-holed sphere subsurface; see Margalit \cite[Lemma 8]{Mar}).  To prove Theorem~\ref{T:Main for S0n}, we thus need only prove the following lemma.

\begin{lemma}
Let $\phi:X_{0,5}\to \MP({S_{g,n}})$ be an injective simplicial map. Then $\phi$ maps the core pentagon $Z_{0,5}\subset X_{0,5}$ to an alternating pentagon in $\MP(S_{g,n})$. 
\end{lemma}
\begin{proof}
The subgraph $\phi(Z_{0,5})$ is a pentagon in $\MP(S_{g,n})$. The triangles attached to the core pentagon $Z_{0,5}$ prevent $\phi(Z_{0,5})$ from being contained in a single Farey graph. By Lemma~\ref{L:no alternating 2-tuple}, there are no alternating $2$- and $3$-tuples, so $\phi(Z_{0,5})$ cannot be contained in two or three different Farey graphs. Suppose $\phi(Z_{0,5})$ is contained in four different Farey graphs. Then four vertices in the subgraph $\phi(Z_{0,5})$ form an alternating $4$-tuple, which must be as in Figure~\ref{F:impossible square0}, up to symmetry. From the figure, there is an edge connecting $\phi(v_2)$ and $\phi(v_3)$ but no edge connecting $\phi(v_1)$ and $\phi(v_4)$, which is a contradiction to Lemma~\ref{L:square}. We conclude that $\phi(Z_{0,5})$ is an alternating pentagon.
\end{proof}

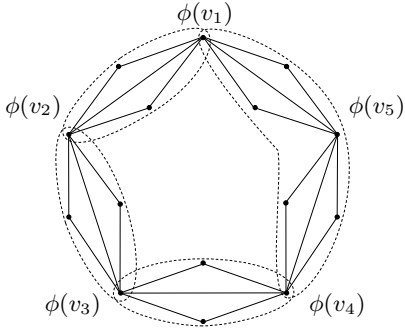
\begin{figure}[ht]
\begin{center}
\includegraphics[height=4cm]{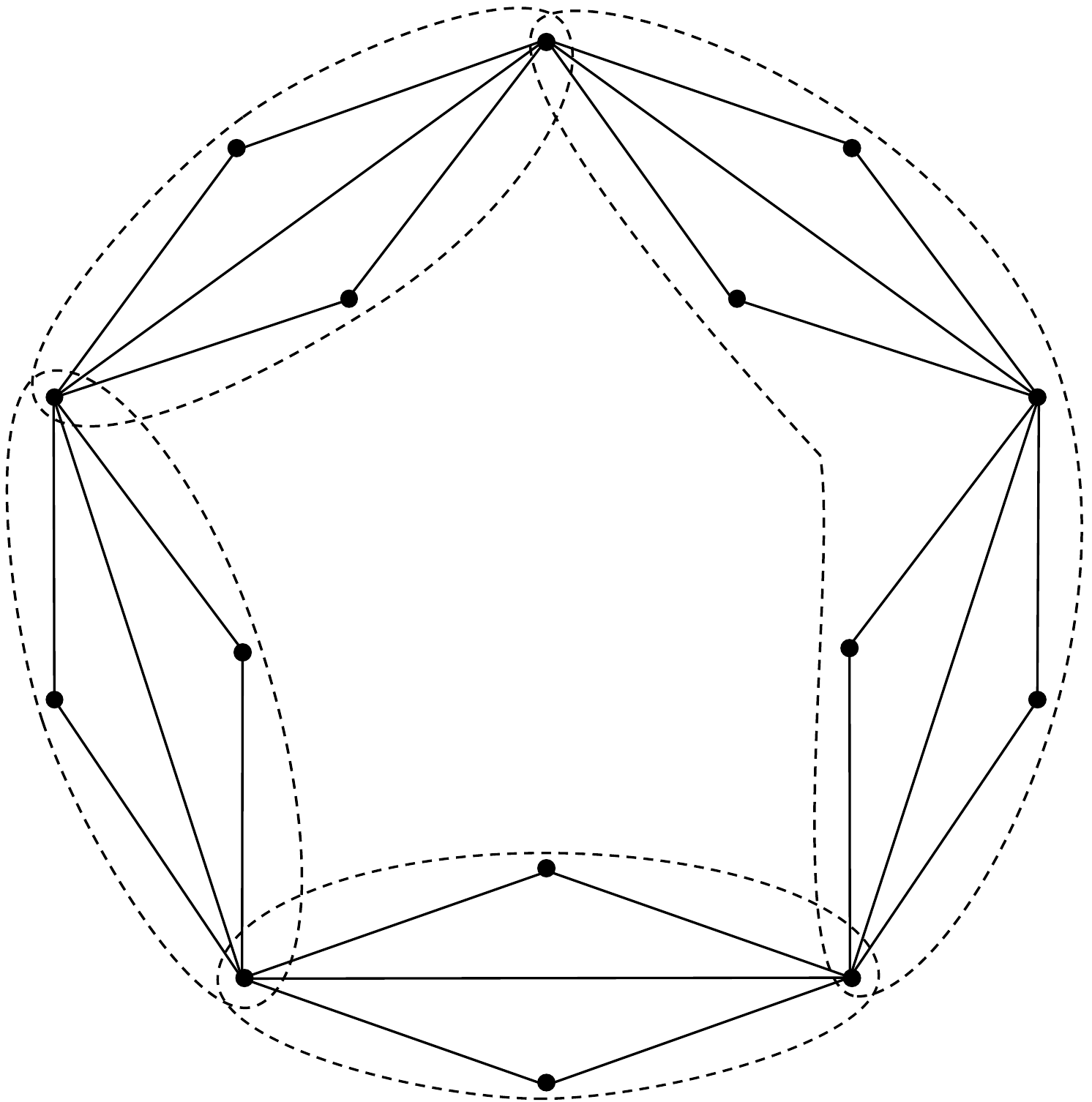} 
\caption{An impossible alternating 4-tuples. Each dashed circle represent a Farey graph.} 
\label{F:impossible square0}
\end{center}

\begin{picture}(0,0)(0,0)
\put(225,160){\small $\phi(v_1)$}
\put(160,125){\small $\phi(v_2)$}
\put(175,50){\small $\phi(v_3)$}
\put(275,50){\small $\phi(v_4)$}
\put(290,125){\small $\phi(v_5)$}
\end{picture}
\end{figure}

The proof of Theorem~\ref{T:Main for S0n} now follows exactly as in \cite{Rasimate1}, and we do not repeat the argument here. \qed

\section{Twice-punctured torus $S_{1,2}$}
\label{sec:S12}
In this section, we construct the finite subgraph $X_{1,2}  \subset \MP(S_{1,2})$ and prove that it is rigid.
We begin by recalling Margalit's definition of the (thickened) almost alternating hexagon $Z_{1,2}$ from~\cite{Mar} and then construct $X_{1,2}$ as a slight enlargement of it.
\begin{definition}
A triple $(v_1,v_2,v_3)$ is called a \textbf{quadrilateral triple} if the three vertices are in a common quadrilateral in a Farey graph and are not in a common triangle. We write the triple so that $v_1$ and $v_3$ are in different triangles and call them the \textbf{outer points} of quadrilateral triple. The vertex $v_2$ is adjacent to both of the endpoints and is called \textbf{the central point} of the quadrilateral triple.
\end{definition}
If $u$ and $v$ are the outer points of a quadrilateral triple, then there is exactly one other quadrilateral triple for which $u$ and $v$ are outer points, and the two central points for the two triples are adjacent to each other.
\begin{definition} An \textbf{almost alternating hexagon} is an alternating $5$-tuple $(v_1,...,v_5)$ having the following properties;
\begin{enumerate}
\item 
$v_i$ and $v_{i+1}$ are adjacent, for $i=1,...,4, and$
\item
$v_1$ and $v_5$ are outer points of a quadrilateral triple.
\end{enumerate}
We also call the subgraph spanned by $v_1,...,v_5$ together with both possible central points of $v_1$ and $v_5$, a \textbf{thickened almost alternating hexagon}.  
\end{definition}
To construct an almost alternating hexagon, we proceed as follows.  Let $\Gamma=\{\beta_1,\beta_2,\beta_3,\beta_4,\beta_5\}$ be the set of curves on $S_{1,2}$ shown in Figure~\ref{F:S12}. The curves in $\Gamma$ generate an almost alternating hexagon $(v_1,...,v_5)$ having  $v_6$ (or $v_7$) as the central point of a quadrilateral triple, as shown in the figure.
Let $Z_{1,2}$ be the subgraph spanned by the vertices $v_i, i=1,...,7$.  By the definition, $Z_{1,2}$ is a thickened almost alternating hexagon. 
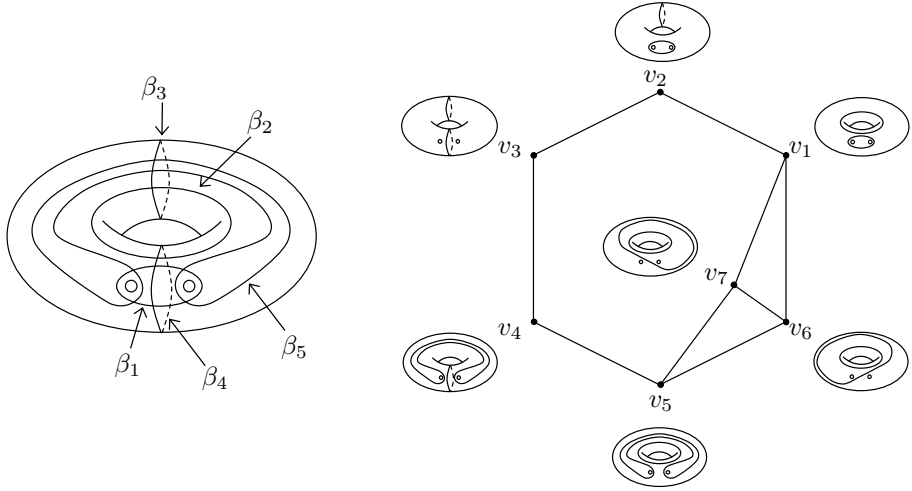
\begin{figure}[ht]
\begin{center}
\includegraphics[height=6.5cm]{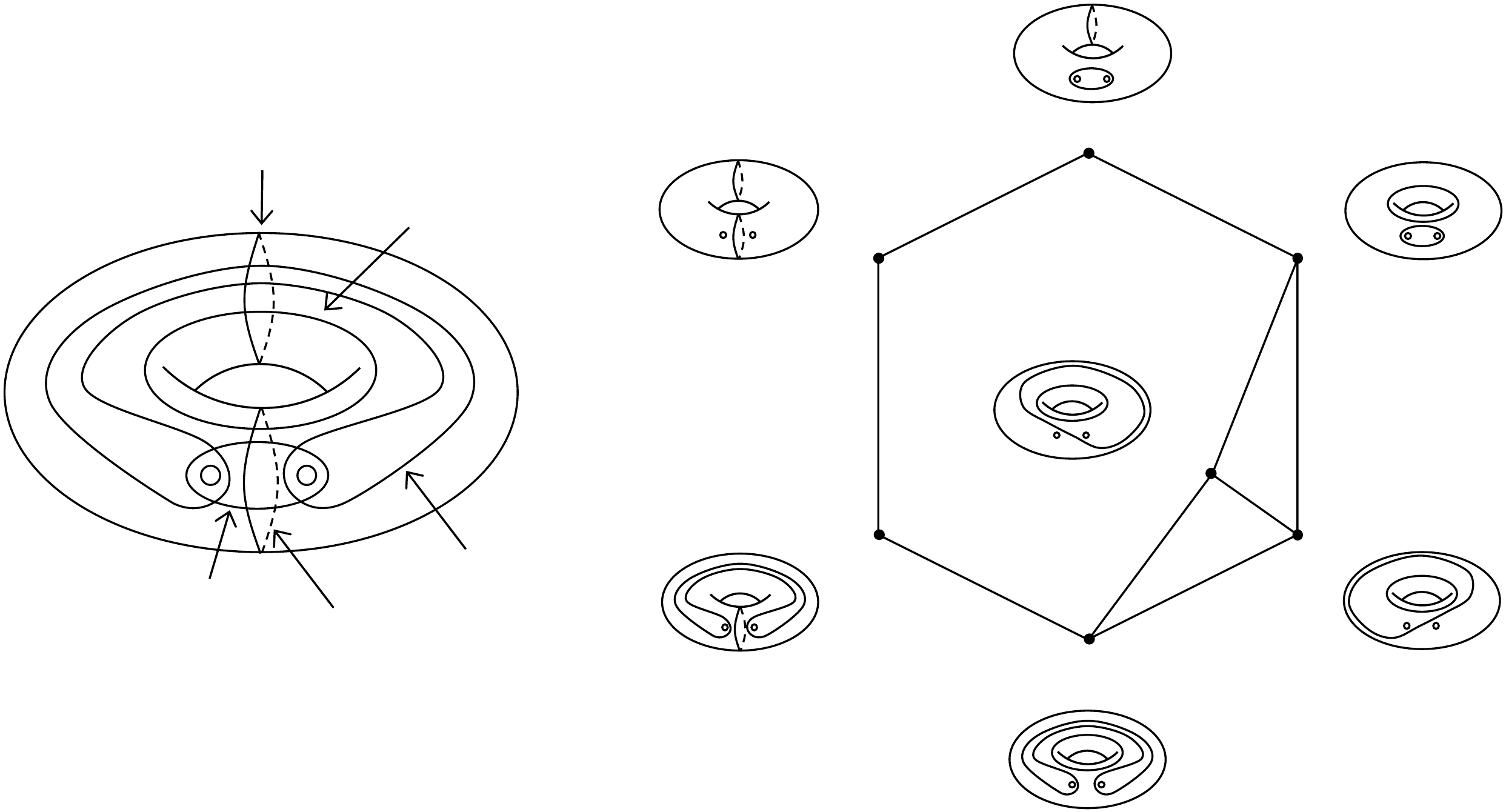} 
\caption{{\bf Left:} A surface $S_{1,2}$ and the curves generating the almost alternating hexagon $Z_{1,2}$.  {\bf Right:} The almost alternating hexagon $Z_{1,2}$.} 
\label{F:S12}
\end{center}
\begin{picture}(0,0)(0,0)
\put(105,100){$\small\beta_1$}
\put(155,192){$\small\beta_2$}
\put(115,205){$\small\beta_3$}
\put(138,95){$\small\beta_4$}
\put(168,107){$\small\beta_5$}
\put(360,183){$\small v_1$}
\put(305,210){$\small v_2$}
\put(250,183){$\small v_3$}
\put(250,114){$\small v_4$}
\put(307,87){$\small v_5$}
\put(360,114){$\small v_6$}
\put(328,135){$\small v_7$}
\end{picture}
\end{figure}
The following lemma, proved by Margalit in \cite[Lemma 9]{Mar}, says that all almost alternating hexagons look like this, up to homeomorphism.
\begin{lemma}
\label{L:almost alternating hexagon} 
Let $Z_{1,2}\subset\MP(S_{1,2})$ be the thickened almost alternating hexagon defined in the previous example and let $\phi:Z_{1,2}\to\MP(S_{g,n})$ be an injective simplicial map. If $\phi(Z_{1,2})$ is a thickened almost alternating hexagon, then there exists a deficiency-2 multicurve $Q$ and a homeomorphism $f \colon S_{1,2}\to (S_{g,n}-Q)_0$ that induces $\phi$. 
\end{lemma} 
The multicurve $Q$ in this lemma is the maximal multicurve in common to all the pants decompositions in the image $\phi(Z_{1,2})$, and hence is uniquely determined by $\phi$.  The next lemma describes the extent to which $f$ is unique.
\begin{lemma} \label{L:f for TAAHs}
The homeomorphism $f$ from Lemma~\ref{L:almost alternating hexagon} is unique up to isotopy and possibly precomposing with the hyperelliptic involution of $S_{1,2}$.
\end{lemma}
\begin{proof} There are five subgraphs $B_1,\ldots,B_5 \subset Z_{1,2}$ contained in five distinct Farey graphs of $\MP(S_{1,2})$ determined by $\beta_1,\ldots,\beta_5$, respectively.  Specifically, $B_i = \MP_{\beta_i}(S_{1,2}) \cap Z_{1,2}$ is the subgraph spanned by vertices that contain $\beta_i$ (so $B_2$ is the union of two triangles, and every other $B_i$ is a single edge).  The unique Farey graph containing $\phi(B_i)$ is the one induced by $Q \cup f(\beta_i)$.  If $f' \colon S_{1,2} \to (S_{g,n}-Q)_0$ is any other homeomorphism with $f'^Q|_{Z_{1,2}} = \phi$, then $f'(\beta_i) = f(\beta_i)$.  Any two such maps $f$ and $f'$ must differ by (an isotopy and) an element of the subgroup $G < \Mod^\pm(S_{1,2})$ consisting of elements that leave each $\beta_i$ invariant.  Since $\beta_2$ and $\beta_3$ intersect in a single point, and $G$ must preserve this intersection point, it is easy to see that $G$ is isomorphic to a subgroup of $\mathbb Z/2 \mathbb Z \times \mathbb Z/2 \mathbb Z$, generated by the hyperelliptic involution and an orientation reversing involution fixing the curves $\beta_3$ and $\beta_4$ pointwise.  The orientation reversing involution switches $v_6$ and $v_7$, though, and so $f$ and $f'$ are either equal, or differ by precomposing by the hyperelliptic involution.
\end{proof}

Our goal now is to produce an enlargement $X_{1,2}$ of $Z_{1,2}$ with the property that for any simplicial embedding $\phi \colon X_{1,2} \to \MP(S_{g,n})$, the subgraph $\phi(Z_{1,2})$ is a thickened almost alternating hexagon. The definition of $X_{1,2}$ is as follows. 

We start with $Z_{1,2}$ in Figure~\ref{F:S12}. Next, we apply half-twists $T^{\pm 1}_{\beta_1}$ and $T^{\pm 1}_{\beta_5}$ to $Z_{1,2}$. The four images of $Z_{1,2}$ under these half-twists are shown in Figure~\ref{F:hexagon1}. Each of the four images of $v_3$ under these twists is connected to $v_3$ by an edge {\em not} contained in $Z_{1,2}$ or any of its images under these four half twists.  We define the subgraph $X_{1,2}$ to be the union of $Z_{1,2}$ with its four images under the twists, and the four edges from $v_3$ to each of its four images; see Figure~\ref{F:X12}.  Note that there are $14$ triangles in $X_{1,2}$ contained in three Farey graphs: one {\bf large Farey graph} determined by $v_1,v_5$, and two {\bf small Farey graphs} determined by $v_2,v_3$ and $v_3,v_4$. 
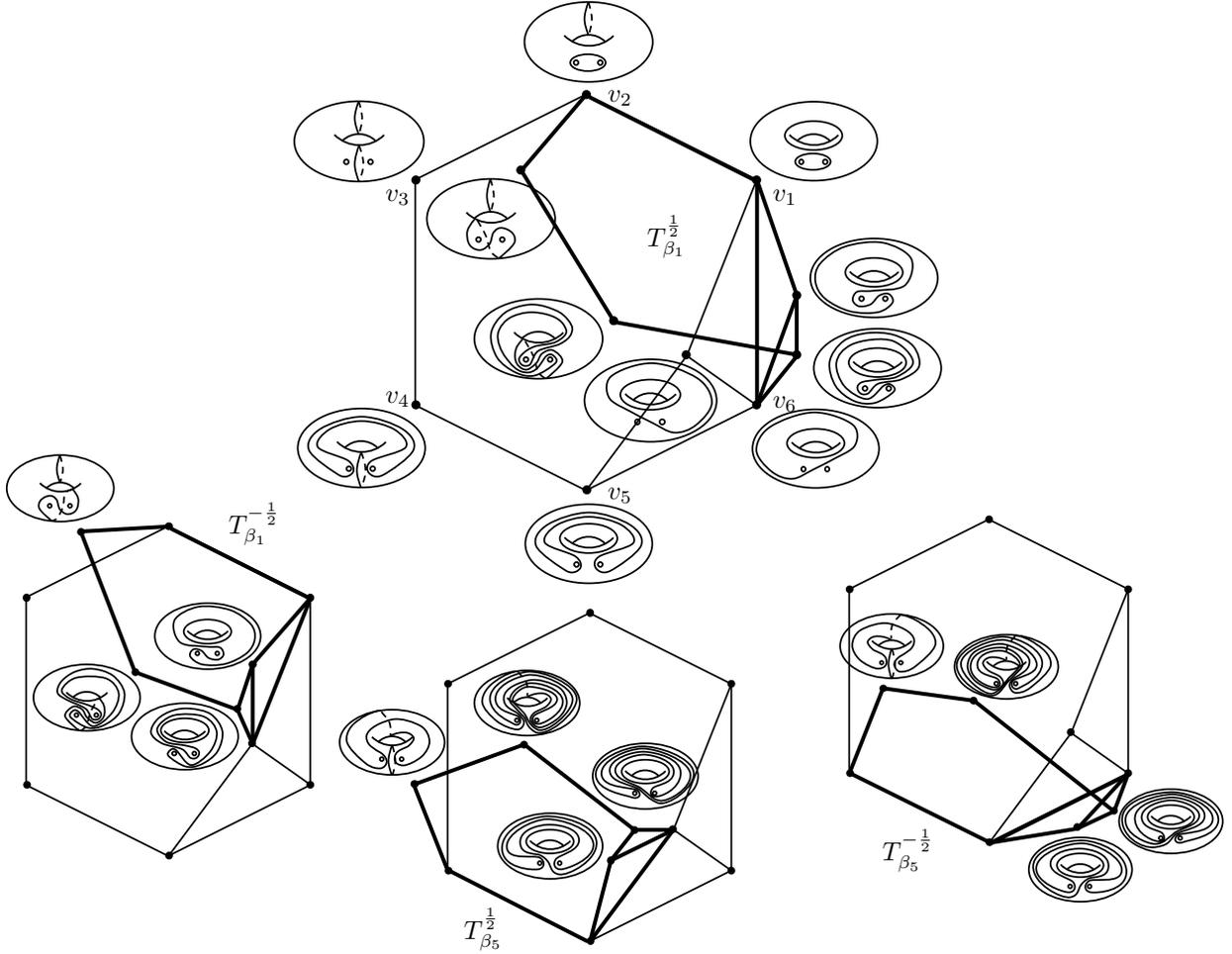
\begin{figure}[htb]
\begin{center}
\begin{tikzpicture}
\node at (-1,0) {\includegraphics[width=16.5cm]{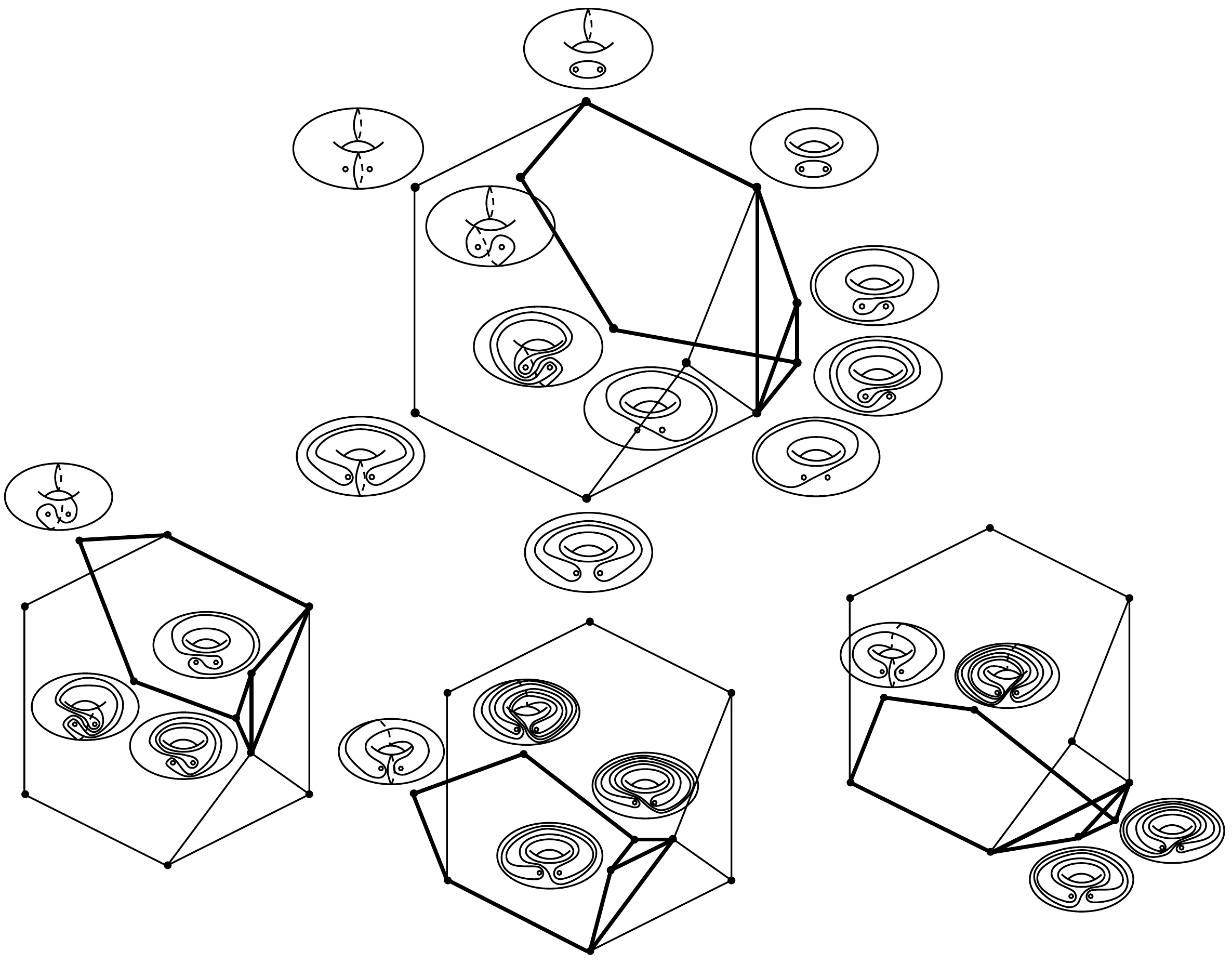}};
\end{tikzpicture}
\caption{{\bf Top:} The twist image of the thickened almost alternating hexagon $Z_{1,2}$ (shown in thick lines) under a half-twist around $\beta_1$. {\bf Bottom:} The other twist images of $Z_{1,2}$ under half-twists around $\beta_1$ and $\beta_5$.}
\label{F:hexagon1}
\end{center}
\begin{picture}(0,0)(0,0)
\put(298,347){$v_1$}
\put(235,385){$v_2$}
\put(150,347){$v_3$}
\put(150,270){$v_4$}
\put(235,233){$v_5$}
\put(298,268){$v_6$}
\put(250,330){$T^{\frac{1}{2}}_{\beta_1}$}
\put(90,220){$T^{-\frac{1}{2}}_{\beta_1}$}
\put(180,65){$T^{\frac{1}{2}}_{\beta_5}$}
\put(340,95){$T^{-\frac{1}{2}}_{\beta_5}$}
\end{picture}
\end{figure}

\begin{figure}[htb]
\begin{center}
\includegraphics[height=5cm]{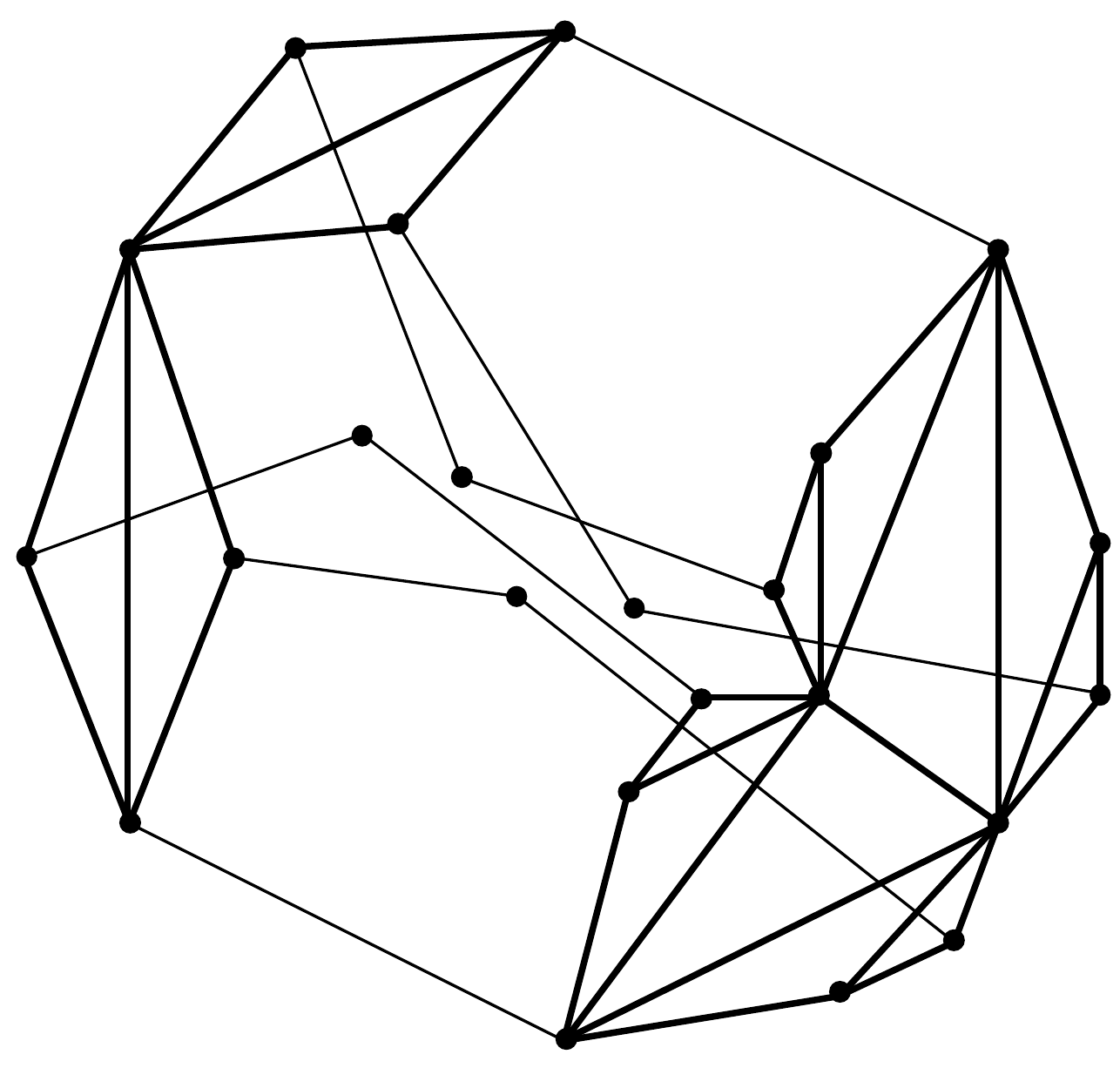} 
\caption{The finite rigid subgraph $X_{1,2}$. Two small Farey graphs (each contains two triangles) and one large Farey graph (contains ten triangles) are shown thickened.}
\label{F:X12}
\end{center}
\end{figure}

We begin by showing that any thickened almost alternating hexagon uniquely determines a subgraph isomorphic to $X_{1,2}$ containing it.
\begin{lemma}\label{L:uniqueness of X_{1,2}}
If $Z\subset\MP(S_{g,n})$ is a thickened almost alternating hexagon then there exists a unique subgraph $X$ containing $Z$, so that $(X,Z) \cong (X_{1,2},Z_{1,2})$.
\end{lemma}

\begin{proof}  By Lemma~\ref{L:almost alternating hexagon}, there is a deficiency-$2$ multicurve $Q \subset S_{g,n}$ and homeomorphism $f \colon S_{1,2} \to (S_{g,n} -Q)_0$, so that $X = f^Q(X_{1,2})$ is a subgraph of $\MP(S_{g,n})$ containing $Z$.  Then $f^Q|_{X_{1,2}}$ defines a homeomorphism of pairs $(X_{1,2},Z_{1,2})$ to $(X,Z)$.  Suppose that $X'$ is another subgraph containing $Z$ with $(X',Z) \cong (X_{1,2},Z_{1,2})$.

First, note that each edge of $\MP(S_{g,n})$ is contained in a unique Farey graph, and in particular is contained in a unique pair of triangles.  Therefore, the $14$ triangles in $X$ (and hence $X'$) are uniquely determined by $Z$.  It follows that $X \cap X'$ contains these $14$ triangles, and so to complete the proof, the remaining four length-$2$ geodesics of $X$ must be shown to be the same as those of $X'$.  

Now set
\[ x_i^{\pm} = f^Q(T_{\beta_1}^{\pm \frac12}(v_i)) \mbox{ and } y_j^{\pm} = f^Q(T_{\beta_5}^{\pm \frac 12}(v_j)) \]
for $i = 3,4,5$ and $j = 1,2,3$.  Then $\{x_3^+,x_4^+,x_5^+\}$, $\{x_3^-,x_4^-,x_5^-\}$,$\{y_3^+,y_2^+,y_1^+\}$ and $\{y_3^-,y_2^-,y_1^-\}$ are the vertices of the four length-$2$ geodesics in $X$, outside the $14$ triangles.
According to Lemma~\ref{L:length2}, these geodesics are unique.  Therefore, if the length-$2$ geodesics in $X'$ also connect $x_3^+$ to $x_5^+$, $x_3^-$ to $x_5^-$, $y_3^+$ to $y_1^+$, and $y_3^-$ to $y_1^-$,  then $X' = X$ as required.  

Since $X' \cong X$, it suffices to show that there are no length-$2$ paths connecting $x_3^+$ to $x_5^-$, nor any connecting $y_3^+$ to $y_1^-$.  In the first case, we write
\[ x_3^+ = Q \cup \{\alpha,\beta\} \mbox{ and } x_5^- = Q \cup \{ \alpha',\beta'\}.\]
Since $f^Q$ is induced by the homeomorphism $f$, inspection of Figure~\ref{F:hexagon1} shows that each of $\alpha,\beta$ intersects each of $\alpha',\beta'$, and so there is no length two path between $x_3^+$ and $x_5^-$.  A similar inspection shows that $y_3^+$ is not connected by $y_1^-$.  Thus, $X' = X$, and we are done.
\end{proof}

We prove the following theorem which is our main theorem for the case of twice-punctured torus.
\begin{theorem} [Main Theorem for $S_{1,2}$] \label{T:Main for S12}
For any surface $S_{g,n}$ and any injective simplicial map
\[
\phi:X_{1,2}\to\MP(S_{g,n}),
\]
there exists a deficiency-$2$ multicurve $Q$ and a homeomorphism $f:S_{1,2}\to (S_{g,n}-Q)_0$ so that $f^Q|_{X_{1,2}} = \phi$.  Moreover, $f$ is unique up to isotopy and pre-composing with the hyperelliptic involution.
\end{theorem}
\begin{remark} Although there is the ambiguity of precomposing with the hyperelliptic involution of $S_{1,2}$, $f$ is sufficiently determined that for any vertex $v$ of $X_{1,2}$ and either curve $\gamma$ in the pants decomposition of $S_{1,2}$ determined by $v$, $f(\gamma)$ is a uniquely determined curve in the pants decomposition $\phi(v)$ of $S_{g,n}$.
\end{remark}
\begin{proof}
We begin the proof by showing that $\phi(Z_{1,2})$ is a thickened almost alternating hexagon. We first claim that $\phi$ cannot map $Z_{1,2}$ into a single Farey graph.  To see this, observe that the quadrilateral spanned by the two quadrilateral triples has the property that any path in the Farey graph connecting the two outer points necessarily passes through at least one of the central points.  Therefore, the $\phi$-image of the path of length four in $Z_{1,2}$ outside the quadrilateral contradicts injectivity if $\phi(Z_{1,2})$ lies in a single Farey graph, proving the claim. 

Moreover, $\phi(Z_{1,2})$ cannot be contained in only two or three Farey graphs since this would produce an alternating $2$--tuple or $3$--tuple by Lemma~\ref{L:circuits to alternating tuples}, which is impossible according to Lemma~\ref{L:no alternating 2-tuple}. 

Therefore, we suppose $\phi(Z_{1,2})$ is contained in four different Farey graphs. Then by Lemma~\ref{L:circuits to alternating tuples} four vertices from $\phi(v_1),...,\phi(v_5)$ form an alternating $4$-tuple, and moreover, the $14$ triangles of $X_{1,2}$ must also map into these four Farey graphs. Up to symmetry, there are three possible cases, which are:
\[ (\phi(v_1),\phi(v_3), \phi(v_4),\phi(v_5)) , \quad\quad
(\phi(v_1),\phi(v_2),\phi(v_4),\phi(v_5)), \quad \mbox{ or } \quad \quad (\phi(v_1),\phi(v_2),\phi(v_3),\phi(v_4)).\]
See Figure~\ref{F:impossible square}.  By inspection, Lemma~\ref{L:square} shows that all these configurations lead to a contradiction.  For example, in the first case there is no edge connecting $\phi(v_1)$ and $\phi(v_3)$ in the Farey graph containing them while there is an edge connecting $\phi(v_4)$ and $\phi(v_5)$, contradicting to Lemma~\ref{L:square}.  Therefore, $\phi(Z_{1,2})$ cannot be in four different Farey graphs and we conclude that $\phi(Z_{1,2})$ is a thickened almost alternating hexagon. 

By Lemma~\ref{L:almost alternating hexagon} there is a  deficiency-$2$ multicurve $Q \subset S_{g,n}$ and a homeomorphism $f \colon S_{1,2} \to (S_{g,n}-Q)_0$ with $f^Q|_{Z_{1,2}} = \phi|_{Z_{1,2}}$, unique up to precomposing with the hyperelliptic involution.  We further observe that $f^Q|_{X_{1,2}}$ induces a graph-isomorphism of the pair $(X_{1,2},Z_{1,2})$ to $(f^Q(X_{1,2}),f^Q(Z_{1,2})) = (f^Q(X_{1,2}),\phi(Z_{1,2}))$.  By Lemma~\ref{L:uniqueness of X_{1,2}}, $f^Q(X_{1,2}) = \phi(X_{1,2})$.  Therefore, $f^Q|_{X_{1,2}}$ and $\phi$ differ by a symmetry of the graph $X_{1,2}$ that is the identity on $Z_{1,2}$.  It is easy to see that such a graph symmetry must be the identity on the vertices, and hence $\phi = f^Q|_{X_{1,2}}$.  This completes the proof.
\end{proof} 
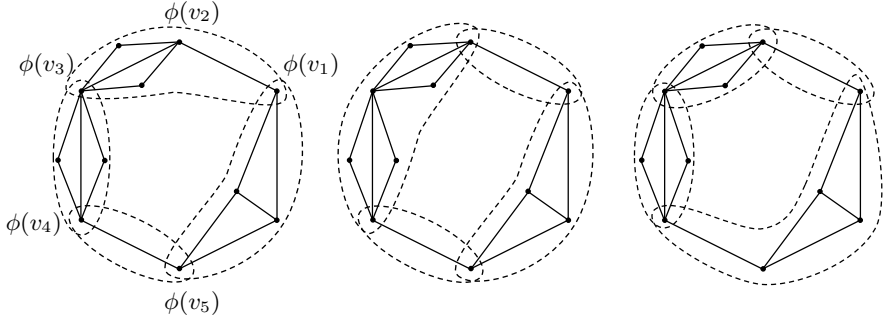
\begin{figure}[ht]
\begin{center}
\includegraphics[height=3.5cm]{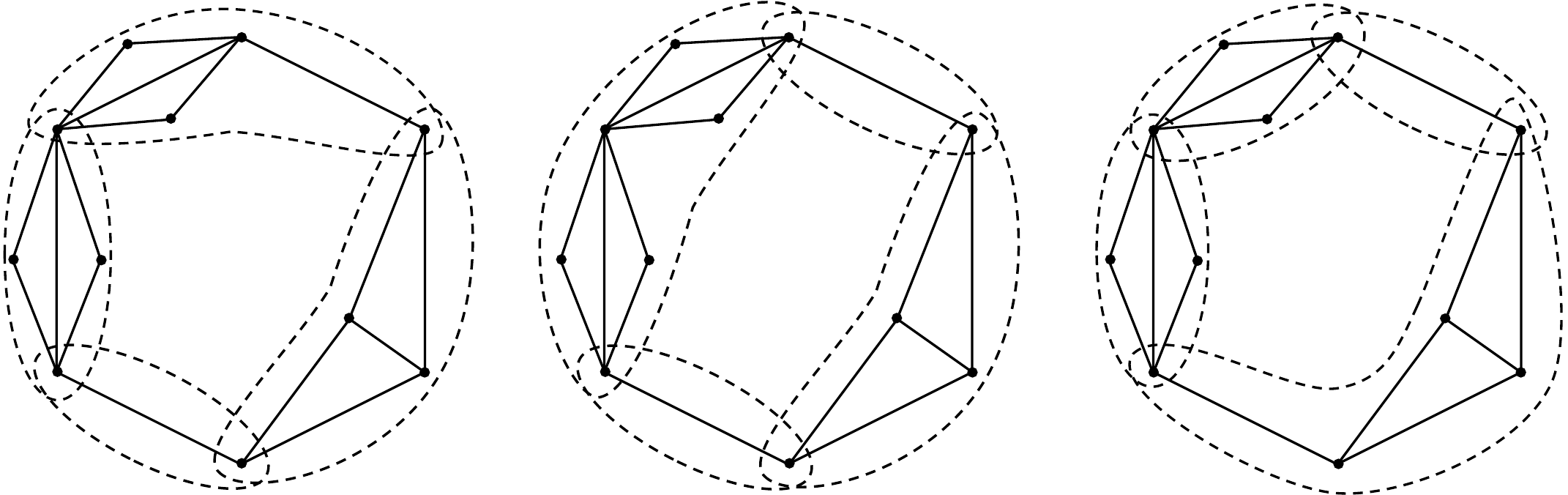} 
\caption{Three images of impossible alternating 4-tuples. Each dashed circle represent a Farey graph.}
\label{F:impossible square}
\end{center}
\begin{picture}(0,0)(0,0)
\put(165,125){\small$\phi(v_1)$}
\put(120,145){\small$\phi(v_2)$}
\put(65,125){\small$\phi(v_3)$}
\put(60,65){\small$\phi(v_4)$}
\put(120,35){\small$\phi(v_5)$}
\end{picture}
\end{figure}

\section{Genus-two surface $S_{2,0}$}
\label{sec:S20}
We next consider the case of a closed genus $2$ surface, $S_{2,0}$.  This follows a different pattern than the other surfaces, so we handle this case separately.  First, consider the three  non-separating, pairwise-disjoint curves $\gamma_1,\gamma_2,\gamma_3$ on $S_{2,0}$  shown on the left in Figure~\ref{F:X20}.  The complement of each $\gamma_i$ is a surface homeomorphic to $S_{1,2}$.  Inside $\MP(S_{2,0})$, the subgraphs $\MP(S_{2,0})_{\gamma_i}$, for $i=1,2,3$, are three copies of $\MP(S_{1,2})$.  Inside each of these, we choose a copy of $Z_{1,2}$ as in Figure~\ref{F:X20} on the right, and then enlarge it to a copy of $X_{1,2}$ as in the previous section.  We denote these latter subgraphs by $X_i \subset \MP(S_{2,0})_{\gamma_i}$, for $i=1,2,3$.  Now set
\[
X_{2,0}=X_1\cup X_2\cup X_3.
\]
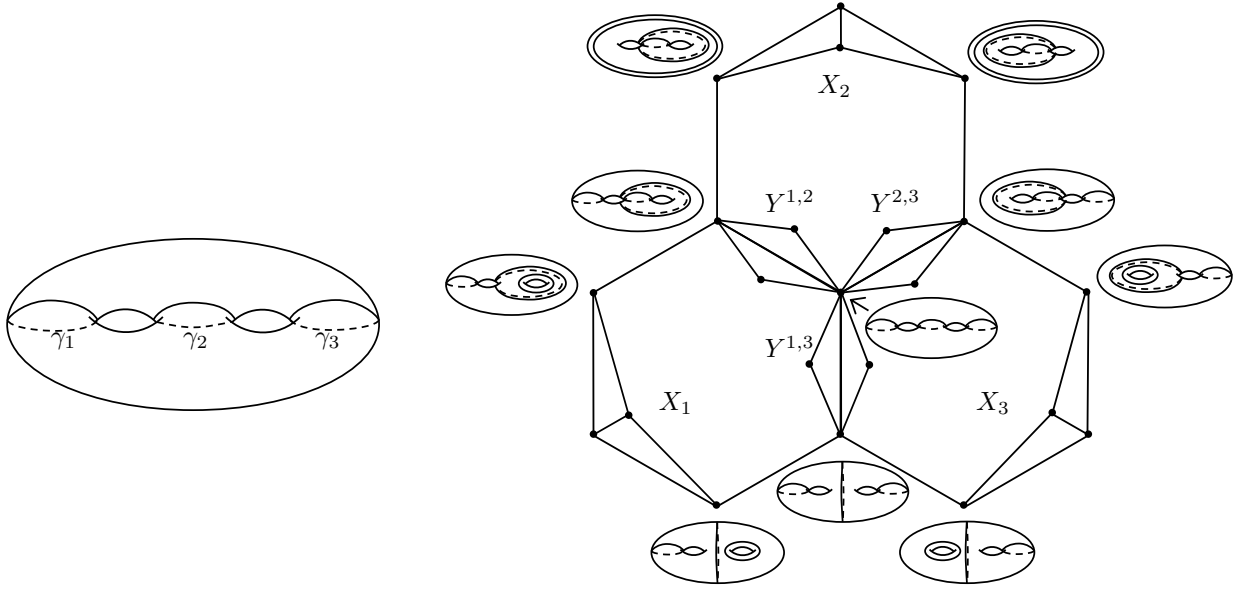
\begin{figure}[ht]
\begin{center}
\includegraphics[height=8cm]{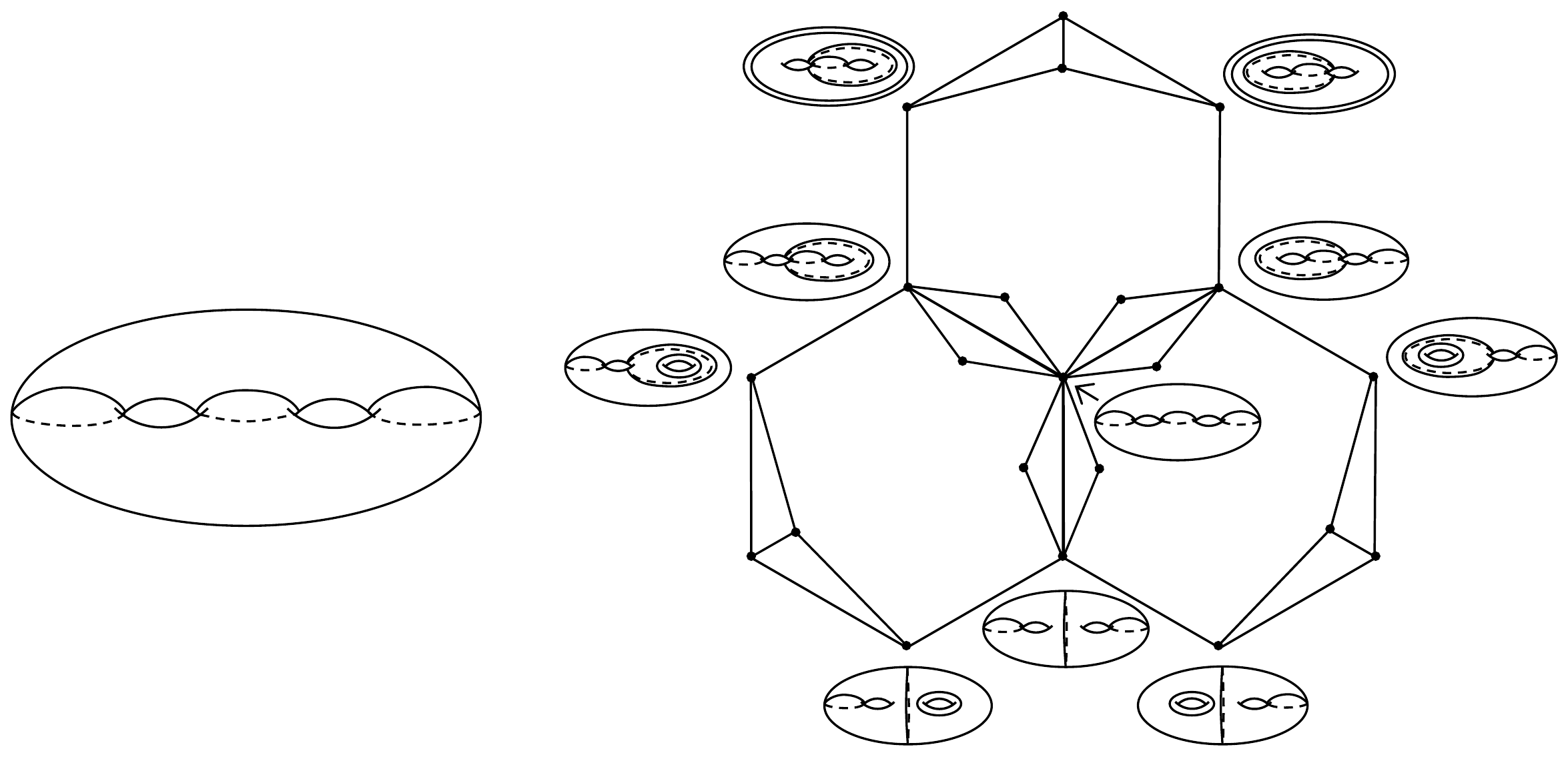}
\caption{{\bf Left:} The three curves $\gamma_1,\gamma_2,\gamma_3$ on $S_{2,0}$.  {\bf Right:} The finite subgraph $X_{2,0}$ is the union of three copies of $X_{1,2}$ as indicated on the right of the figure (only the defining $Z_{1,2}$ subgraph and the part of the Farey graph of each copy of $X_{1,2}$ is shown to avoid complicating the picture). Each copy comes from a subsurface homeomorphic to $S_{1,2}$ obtained by cutting $S_{2,0}$ along one of the curves $\gamma_1$, $\gamma_2$, and $\gamma_3$.}
\label{F:X20}
\end{center}
\begin{picture}(0,0)(0,0)
\put(20,175){$\small\gamma_1$}
\put(70,175){$\small\gamma_2$}
\put(120,175){$\small\gamma_3$}
\put(370,150){$X_3$}
\put(250,150){$X_1$}
\put(310,270){$X_2$}
\put(290,225){$Y^{1,2}$}
\put(290,170){$Y^{1,3}$}
\put(330,225){$Y^{2,3}$}
\end{picture}
\end{figure}
\begin{theorem} [Main Theorem for $S_{2,0}$] \label{T:Main for S20}
Suppose $S_{g,n}$ is any surface, and
\[
\phi:X_{2,0}\to\MP(S_{g,n}),
\]
is a simplicial embedding.  Then $(g,n) = (2,0)$ and there is a homeomorphism $f \colon S_{2,0} \to S_{2,0}$ that induces $\phi$.  The homeomorphism $f$ is unique up to isotopy and composing with the hyperelliptic involution.
\end{theorem}

\begin{proof}  According to Theorem~\ref{T:Main for S12}, there are three deficiency-$2$ multicurves $Q_1,Q_2,Q_3 \subset S$ and homeomorphisms $f_i \colon (S_{2,0} - \gamma_i) \to (S_{g,n}-Q_i)_0$, for $i=1,2,3$, so that $\phi|_{X_i} = f_i^{Q_i}|_{X_i}$.  

For each $i \neq j$, let $Y^{i,j} = X_i \cap X_j$ denote the intersection, which is contained in the Farey graph determined by the multicurve $\gamma_i \cup \gamma_j$.  By definition, $Y^{i,j} = Y^{j,i}$.
For $i,j,k$ all distinct, we note that $f^{Q_i}_i$ sends each of the Farey graphs containing $Y^{i,j}$ and $Y^{i,k}$ to distinct Farey graphs.  Therefore, the three Farey graphs containing $\phi(Y^{1,2}),\phi(Y^{1,3}),$ and $\phi(Y^{2,3})$ in $\MP(S_{g,n})$ are all distinct, and hence there is a deficiency-$3$ multicurve $Q$ so that for all $i \neq j$, we have
\[ Q = Q_i \cap Q_j = Q_1 \cap Q_2 \cap Q_3. \]
Therefore, for each $i=1,2,3$, there is some $\beta_i$ so that $Q_i = Q \cup \{\beta_i\}$, and $\beta_i \neq \beta_j$, for $i \neq j$ (c.f. Aramayona's argument in~\cite[Theorem C]{Aramayona}).

On the other hand, $\phi(Y^{1,2}) = f^{Q_1}_1(Y^{1,2}) = f^{Q_2}_2(Y^{1,2})$ is contained in the Farey graph that is determined by both $Q_1 \cup f^{Q_1}_1(\gamma_2)$ as well as $Q_2 \cup f^{Q_2}_2(\gamma_1)$.  Combining this with the previous paragraph, we see that $f_1(\gamma_2) = \beta_2$ and $f_2(\gamma_1) = \beta_1$.  Considering all permutations of indices and arguing similarly, we have
\begin{equation} \label{E:S20betas}
f_1(\gamma_2) = \beta_2 = f_3(\gamma_2), \, \, f_1(\gamma_3) = \beta_3 = f_2(\gamma_3), \mbox{ and } f_2(\gamma_1) = \beta_1 = f_3(\gamma_1).
\end{equation}

Now consider the two pairs of pants $P_+ \cup P_- = S_{2,0} - (\gamma_1 \cup \gamma_2 \cup \gamma_3)$.  Equation~(\ref{E:S20betas}) implies that $\beta_2$ and $\beta_3$ are boundary components of {\em both} pairs of pants $f_1(P_\pm)$ in $S_{g,n} - (Q \cup \beta_1 \cup \beta_2 \cup \beta_3)$.
Similarly, the pants $f_2(P_\pm) \subset S_{g,n} - (Q \cup \beta_1 \cup \beta_2 \cup \beta_3)$ have $\beta_1$ and $\beta_3$ as boundary components.  These must be the same pants as $f_1(P_\pm)$, and hence $\beta_1,\beta_2$, and $\beta_3$ are all boundary components of the same two pairs of pants in $S_{g,n} - (Q \cup \beta_1 \cup \beta_2 \cup \beta_3)$.  But that means that the union of these two pairs of pants is a genus two surface in $S_{g,n}$.  This is only possible if $S_{g,n} = S_{2,0}$, and $Q = \emptyset$.

Therefore, $Q_i = \{\beta_i\}$, for $i = 1,2,3$, and $f_i \colon S_{2,0} - \gamma_i \to S_{2,0} - \beta_i$ is a homeomorphism.  The hyperelliptic involution of $S_{2,0}$ interchanges the two pairs of pants $P_\pm$, and leaves invariant every curve.  Therefore we may precompose some of the maps $f_i$ with the hyperelliptic involution so that
\[ f_1(P_+) = f_2(P_+) = f_3(P_+)  \mbox{ and }  f_1(P_-) = f_2(P_-) = f_3(P_-). \]
Moreover, appealing to (\ref{E:S20betas}), we see that on $P_+$ and $P_-$, the maps $f_1$, $f_2$, and $f_3$ agree up to isotopy, and possibly an orientation reversing involution preserving each boundary component.  On the other hand, $f_1^{Q_1}$ and $f_2^{Q_2}$ agree on $Y^{1,2}$, which contains a triangle of the Farey graph defined by $\gamma_1 \cup \gamma_2$.  Consequently, $f_1$ and $f_2$ agree on every simple closed curve on $S_{2,0} - (\gamma_1 \cup \gamma_2)$.    The same is true for $f_1$ and $f_3$ on $S_{2,0} - (\gamma_1 \cup \gamma_3)$ and $f_2$ and $f_3$ on $S_{2,0} - (\gamma_2 \cup \gamma_3)$.  From this we deduce that all maps must either preserve or all maps must reverse the orientation on $P_+$ and $P_-$, and hence all three maps agree up to isotopy on $P_+$ and $P_-$.  This also implies that on each $i \neq j$, $f_i$ and $f_j$ restricted to $S_{2,0} - (\gamma_i \cup \gamma_j)$ agree up to (isotopy and) precomposing with one of the finite number of hyperelliptic mapping classes of the $S_{2,0} - (\gamma_i \cup \gamma_j)$ that act as the identity on the set of isotopy classes of essential simple closed curves.  Since $f_i$ and $f_j$ agree on the boundary components of this, in fact they must agree up to isotopy.  It follows that $f_1$, $f_2$, and $f_3$ {\em glue together} to well-define a mapping class $f \colon S_{2,0} \to S_{2,0}$ that agrees with $f_i$ on $S_{2,0} - \gamma_i$, for each $i = 1,2,3$.  

Now, for any $v \in X_i$, write $v = \gamma_i \cup \alpha \cup \beta$, and recall that $Q_i = \{\beta_i\}$ (since $(g,n) = (2,0)$).  The map $f_i^{Q_i}$ is defined by
\[ \phi(v) = f_i^{Q_i}(v) = \beta_i \cup f_i(\alpha) \cup f_i(\beta).\]
Now note that $\beta_i = f_j(\gamma_i)$, for $j \neq i$, and so this becomes
\[ \phi(v) = f_j(\gamma_i) \cup f_i(\alpha) \cup f_i(\beta) = f(\gamma_i) \cup f(\alpha) \cup f(\beta) = f(\gamma_i \cup \alpha \cup \beta) = f(v).\]
Since $X_{2,0}$ is the union of the subgraphs $X_i$, for $i =1,2,3$, it follows that $f$ induces $\phi$.  The uniqueness of each $f_i$ up to the hyperelliptic involution of $S_{2,0} - \gamma_i$ implies that $f$ is unique up to the hyperelliptic involution of $S_{2,0}$. \end{proof}
\section{General surface}
\label{sec:Sgn}

Let $S = S_{g,n}$ with $g\geq 1$, $\kappa(S) \geq 3$ and $S \neq S_{2,0}$.  In this section, we construct $X_{g,n}$ and prove it is rigid.  First, we recall that a \textbf{cut system} of $S$ is a maximal multicurve $C$ such that $S - C$ is connected; this implies that $C$ is a cut system if and only if $S - C$ is connected and $|C| = g$. By the classification of surfaces, all cut system of $S$ are the one that appears as $\gamma_1,\ldots,\gamma_g$ on the left in Figure~\ref{F:cut system}, up to homeomorphism.

\begin{figure}[ht]
\begin{center}
\includegraphics[height=6cm]{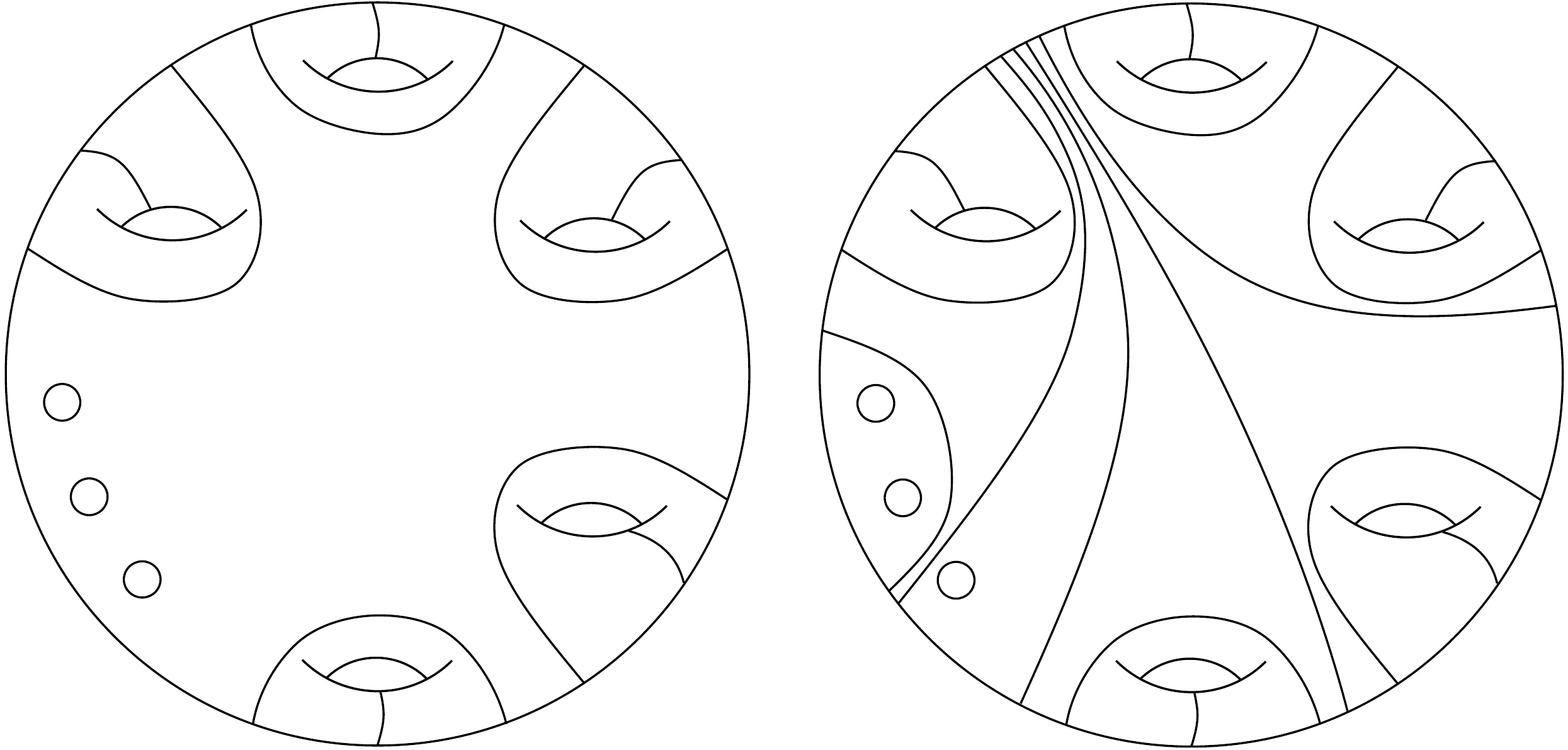} 
\caption{{\bf Left:} $S_{5,3}$ with a cut system $C=\{\gamma_1,...,\gamma_5\}$ and a multicurve $B=\{\beta_1,...,\beta_5\}$. {\bf Right:} $S_{5,3}$ with a pants decomposition $P=C\cup B\cup A$.}
\label{F:cut system}
\end{center}
\begin{picture}(0,0)(0,0)
\put(90,185){$\small\gamma_5$}
\put(145,217){$\small\gamma_4$}
\put(203,180){$\small\gamma_3$}
\put(195,95){$\small\gamma_2$}
\put(133,62){$\small\gamma_1$}
\put(115,155){$\small\beta_5$}
\put(138,183){$\small\beta_4$}
\put(160,158){$\small\beta_3$}
\put(160,117){$\small\beta_2$}
\put(140,93){$\small\beta_1$}
\end{picture}
\end{figure}
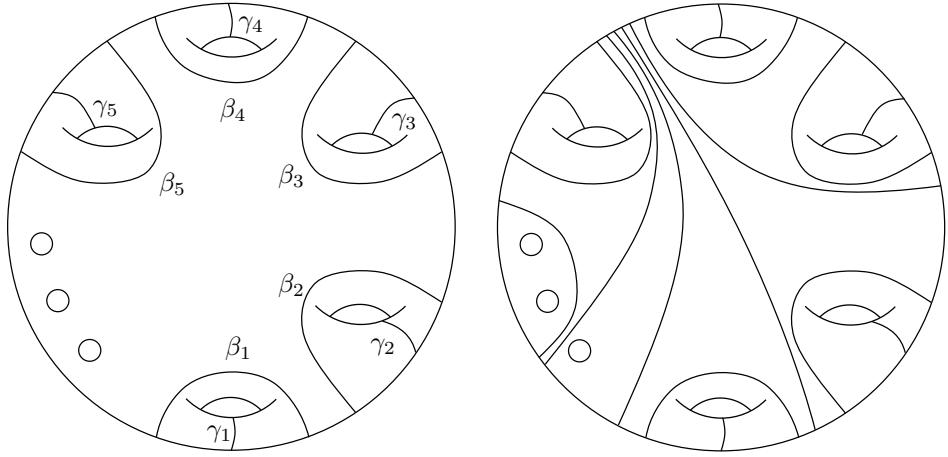
Let $C = \{\gamma_{1}, \ldots, \gamma_{g}\}$ be a fixed cut system of $S$, and let $B = \{\beta_{1}, \ldots, \beta_{g}\}$ be a multicurve such that $C \cup B$ is a multicurve, and for all $i= 1, \ldots, g$, $\beta_{i}$ bounds a one-holed torus which contains $\gamma_{i}$. Note again that up to homeomorphism, all possible choices of $B$ appear as on the left in Figure~\ref{F:cut system}.
We then extend $C \cup B$ to the pants decomposition $P = C \cup B \cup A$.

Let $\Sigma_0 = S - C \cong S_{0,2g+n}$, and observe that $A \cup B$ is a pants decomposition of $\Sigma_0$.  We have an isomorphism $h_0^C \colon \MP(S_{0,2g+n}) \to \MP_C(S_{g,n})$ induced by a homeomorphism $h_0 \colon S_{0,2g+n} \to \Sigma_0$.  This isomorphism sends the vertex defined by $h_0^{-1}(A \cup B)$ to the vertex defined by $P$, and without loss of generality we may assume that $h_0^{-1}(A \cup B)$ defines a vertex of $Z_{0,2g+n}$, the subgraph used in the construction of $X_{0,2g+n}$ (see Section~\ref{sec:Background}).   Now define
\[ Z_0 = h_0^C(Z_{0,2g+n}) \quad \mbox{ and } \quad X_0  = h_0^C(X_{0,2g+n})\]
and note that $Z_0 \subset X_0 \subset \MP_C(S_{g,n})$.


Next, for all $i= 1, \ldots, g$, set $M_{i} = P \backslash \{\gamma_{i},\beta_{i}\}$ and write $\Sigma_i \cong S_{1,2}$ to denote the nontrivial component of $S - M_{i}$.  A homeomorphism $h_i \colon S_{1,2} \cong \Sigma_i$ induces an isomorphism $h_i^{M_i} \colon \MP(S_{1,2}) \to \MP_{M_i}(S_{g,n})$.  Without any further loss in generality, we may assume that $h_i^{M_i}(v_2) = P$ and $h_i^{M_i}([v_2,v_3]) \subset h_0^C(Z_{0,2g+n})$, where $v_2,v_3 \in Z_{1,2}$ are as in Figure~\ref{F:S12}, and $[v_2,v_3]$ denotes the edge spanned by these vertices.  We then let
\[ Z_i = h_i^{M_i}(Z_{1,2}) \quad \mbox{ and } \quad X_i = h_i^{M_i}(X_{1,2}),\]
and note that $Z_i \subset X_i \subset \MP_{M_i}(S_{g,n})$.  See Figures~\ref{F:X13}, \ref{F:X21}, and  \ref{F:X22} for examples.

By construction, there is an edge $e_i \subset Z_0 \cap Z_i = h_0^C(Z_{0,2g+n}) \cap h_i^{M_i}(Z_{1,2})$.  On the other hand, in each of $Z_{0,2g+n}$ and $Z_{1,2}$ the respective edges sent to $e_i$ are contained in a pair of triangles in the respective enlargements $X_{0,2g+n}$ and $X_{1,2}$.  Therefore, for each $i = 1,\ldots, g$, setting $Y_i = X_0 \cap X_i$, we see that $Y_i$ contains the pair of triangles containing $e_i$, contained in the Farey graph $F_i \subset \MP(S_{g,n})$ defined by $e_i$.  In fact, because $X_0 \subset \MP_{\gamma_i}(S_{g,n})$, inspection of Figure~\ref{F:hexagon1} shows that $Y_i$ is exactly the union of the two triangles containing $e_i$.  On the other hand, since $\Sigma_i \cap \Sigma_j$ overlap in at most a pair of pants in the complement of $P$ for any $i,j \geq 1$ with $i \neq j$, it follows that for such $i,j$, $X_i \cap X_j  = \{P\}$.

After these observations, we now define $X_{g,n}$ to be the subgraph
\[ X_{g,n} =  \bigcup_{i=0}^{g} X_{i} \subset \MP(S_{g,n}).\]
Note that $X_{g,n}$ depends on the choice of $P$ (as well as the choices of the various homeomorphisms involved), and Figures~\ref{F:X13},~\ref{F:X21}, and~\ref{F:X22} in the appendix show possible examples of $X_{1,3}, X_{2,1}$ and $X_{2,2}$, respectively.  Now we are ready to prove the following theorem using the same general idea as in the proof of Theorem \ref{T:Main for S20}.

\begin{theorem} [Main Theorem for $S_{g,n}$] 
 Let $S = S_{g,n}$ $g\geq 1$, $\kappa(S) \geq 3$ and $(g,n) \neq (2,0)$. For any surface $S' = S_{g',n'}$ and any injective simplicial map
\[
\phi:X_{g,n}\to\MP(S'),
\]
there exists a deficiency-$\kappa(S)$ multicurve $Q$ and a homeomorphism $f:S \to (S'-Q)_0$ so that $f^Q|_{X_{g,n}} = \phi$. Moreover, $f$ is unique up to isotopy.
\end{theorem}
\begin{proof}
 Let $S^{\prime} = S_{g^{\prime},n^{\prime}}$, and $\phi: X_{g,n} \to \MP(S^{\prime})$ be an injective simplicial map. For each $i = 0, \ldots, g$, we denote by $\phi_{i}$ the restriction of $\phi$ to $X_{i}$. Since each $\phi_{i}$ is injective and $X_{i}$ is a copy of a rigid graph, there exist multicurve $Q_0$ of deficiency $2g+n-3$, and $Q_i$ of deficiency $2$ for $i \geq 1$, together with homeomorphisms
\[ f_i \colon \Sigma_i \to (S'- Q_i)_0, \]
such that $f_i^{Q_i}|_{X_i} = \phi_i$.
The homeomorphism $f_0$ is uniquely determined by $\phi_0$, while for $i \geq 1$, $f_i$ is determined up to pre-composing with the hyperelliptic involution.

For $i = 1,\ldots,g$, bserve that $Y_i= X_0 \cap X_i$ is a union of two triangles containing $e_i$ in the Farey graph $F_i$ of a four-holed sphere $\Sigma^{0,i} \subset S_{g,n}$.  Since
\[ \phi_0|_{Y_i} = \phi|_{Y_i} = \phi_i|_{Y_i},\]
it follows that $f_0|_{\Sigma^{0,i}}$ and $f_i|_{\Sigma^{0,i}}$ agree up to one of the hyperelliptic involutions of $\Sigma^{0,i}$, for each $i = 1,\ldots,g$.  For each $i$, there are two boundary components $\gamma_i^\pm$ of $\Sigma^{0,i}$ that are identified in $S_{g,n}$ to the curve $\gamma_i$.  

We claim that for each $i$, $f_0|_{\Sigma^{0,i}}$ and $f_i|_{\Sigma^{0,i}}$ agree up to the hyperelliptic involution that set-wise preserves the pair of boundary curves $\{\gamma_i^+,\gamma_i^-\}$ (but interchanges them).  
If we let $\Omega \subset \Sigma^{0,i}$ be the pair of pants bounded by $\gamma_i^+$, $\gamma_i^-$, and $\beta_i$, then since $f_0(\beta_i) = f_i(\beta_i)$, what we need to show is that $f_0(\Omega) = f_i(\Omega)$.  Next, note that $\gamma_i^\pm$ are boundary components of $\Sigma_0$, and so $f_0(\gamma_i^\pm)$ are boundary components of $f_0(\Sigma_0)$.  Since $2g+n > 4$ by hypothesis, at least one of the other two boundary components of $f_0(\Omega)$ is an essential curve $\delta_i$ in $f_0(\Sigma_0)$.  Since $f_0(\Sigma_0)$ is a sphere with holes, the pants in $S' - \phi(P)$ adjacent to $\delta_i$ are distinct pairs of pants.  Now if $f_i(\Omega) \neq f_0(\Omega)$, then one of $f_i(\gamma_i^+)$ or $f_i(\gamma_i^-)$ would have to be the essential curve $\delta_i$ in $f_0(\Sigma_0)$.   Since $\gamma_i^\pm$ are identified to $\gamma_i$ in $\Sigma_i$, it follows that $f_i(\gamma_i^+) = \delta_i = f_i(\gamma_i^-)$ in $S'$, and hence $f_i(\Omega)$ is the pair of pants on both sides of $\delta_i$, a contradiction.  This proves the claim.

For each $i \geq 1$, we may precompose $f_i$ with the hyperelliptic involution of $\Sigma_i$ from the previous paragraph, so that $f_0|_{\Sigma^{0,i}} = f_i|_{\Sigma^{0,i}}$.  We can therefore glue together the maps $f_0,f_1,\ldots,f_g$ to a single map $f \colon S \to S'$.  For each of the boundary components $\gamma_i^\pm$, we have $f_0(\gamma_i^\pm) = f_i(\gamma_i)$, and thus $f(\gamma_i)$ is a component of $Q_0$.  Since at most two boundary components of $\Sigma_0$ can map to a single curve in $S'$, it follows that $f(\gamma_1),\ldots,f(\gamma_g)$ are $g$ distinct components of $Q_0$, and hence
\[ Q = Q_0 - f(\gamma_1) \cup \ldots \cup f(\gamma_g) \]
is a multicurve with deficiency $g+ 2g+n-3 = 3g+n-3 = \kappa(S)$ and $f(S) = (S' - Q)_0$.  Since $f_i$ restricts to $f$ on $\Sigma_i$ where $f_i$ induces $\phi_i$ for each $i$, and since $X_{g,n}$ is the union of the subgraphs $X_i$, it follows that $f$ induces $\phi$; that is, $\phi = f^Q|_{X_{g,n}}$, as required.

Uniqueness follows from the uniqueness of the defining maps $f_0,\ldots,f_g$.  Specifically, uniqueness of $f_0$ implies $f$ is unique up to isotopy and possibly Dehn twisting in one or more of $\gamma_1,\ldots,\gamma_g$.  However, uniqueness of $f_1,\ldots,f_g$ (up to the hyperelliptic involutions) implies that there is no Dehn twisting, and thus $f$ is  unique up to isotopy.

\end{proof}


\section{Appendix: Additional examples of finite rigid sets.}
Here we provide three additional examples of the finite rigid sets from Section~\ref{sec:Sgn} to better illustrate the construction.
\begin{figure}[ht]
\begin{center}
\includegraphics[height=9cm]{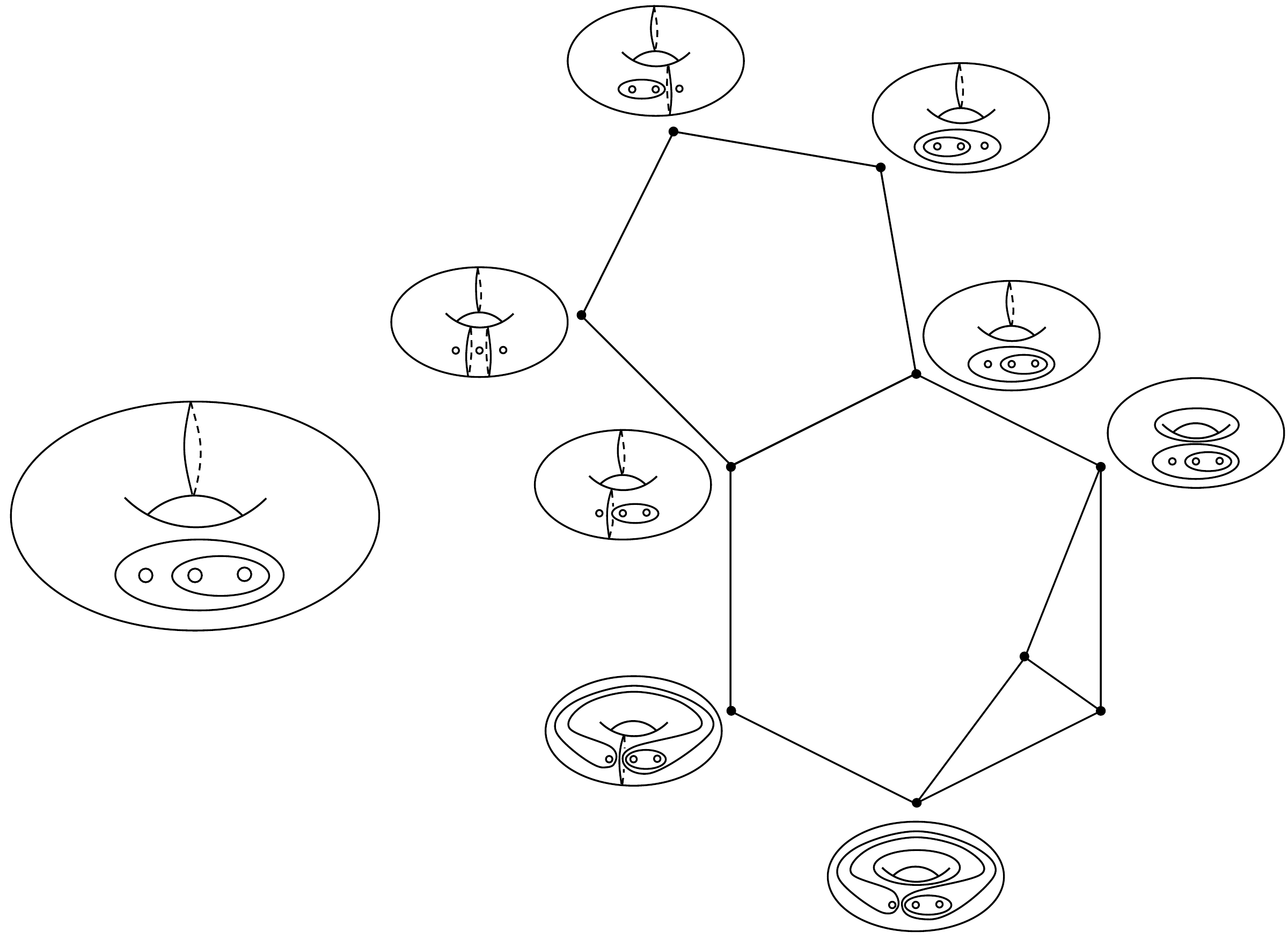} 
\caption{{\bf Left:} $S_{1,3}$ and a pants decomposition $P=C\cup B\cup A$.  {\bf Right:}  Subgraphs $Z_0 \subset X_0 \cong X_{0,5}$ and $Z_1 \subset X_1 \cong X_{1,2}$ glued along $e_1$ used to build $X_{1,3} = X_0 \cup X_1$. Note the vertex $P \in X_0 \cap X_1$.}
\label{F:X13}
\end{center}
\begin{picture}(0,0)(0,0)
\put(81,155){$\small\beta_1$}
\put(100,185){$\small\gamma_1$}
\put(306,195){$\small P$}
\put(260,230){$\small Z_0$}
\put(300,150){$\small Z_1$}
\put(283,188){$\small e_1$}
\end{picture}
\begin{center}
\includegraphics[height=8.5cm]{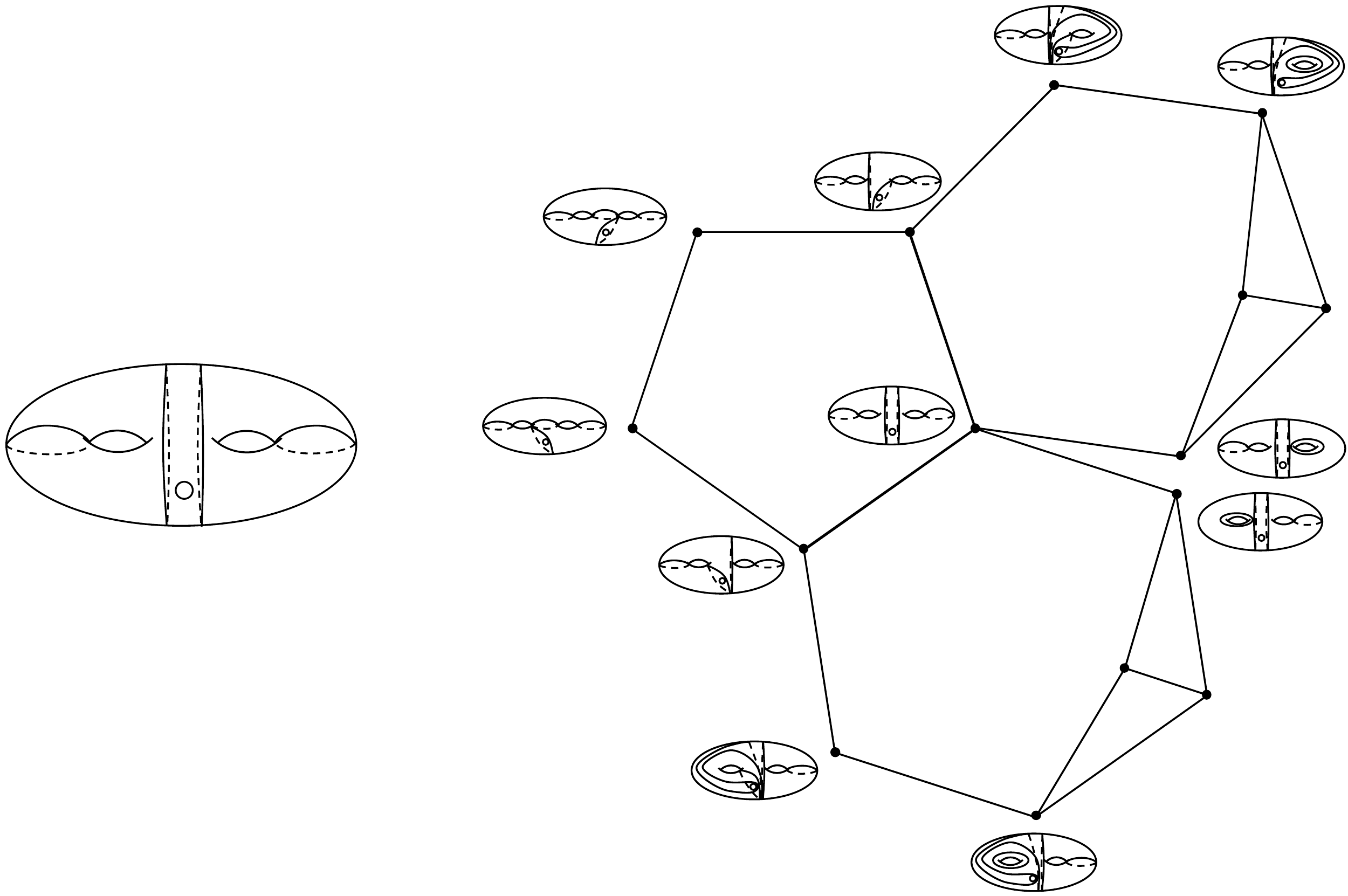} \caption{{\bf Left:} $S_{2,1}$ and a pants decomposition $P=C\cup B$. {\bf Right:} The subgraphs $Z_0 \subset X_0 \cong X_5$ and $Z_i \subset X_i \cong X_{2,1}$, for $i = 1,2$,  glued to $X_0$ along edges $e_1$ and $e_2$ used to construct $X_{2,1}=X_0\cup X_1\cup X_2$.  Note the vertex $P$ in $X_0 \cap X_1 \cap X_2$.}
\label{F:X21}
\end{center}
\begin{picture}(0,0)(0,0)
\put(85,175){$\small\beta_1$}
\put(110,175){$\small\beta_2$}
\put(42,185){$\small\gamma_1$}
\put(152,185){$\small\gamma_2$}
\put(313,180){$\small P$}
\put(263,210){$\small Z_0$}
\put(315,140){$\small Z_1$}
\put(345,235){$\small Z_2$}
\put(290,167){$\small e_1$}
\put(310,225){$\small e_2$}
\end{picture}
\end{figure}
\begin{figure}[ht]
\begin{center}
\includegraphics[height=10cm]{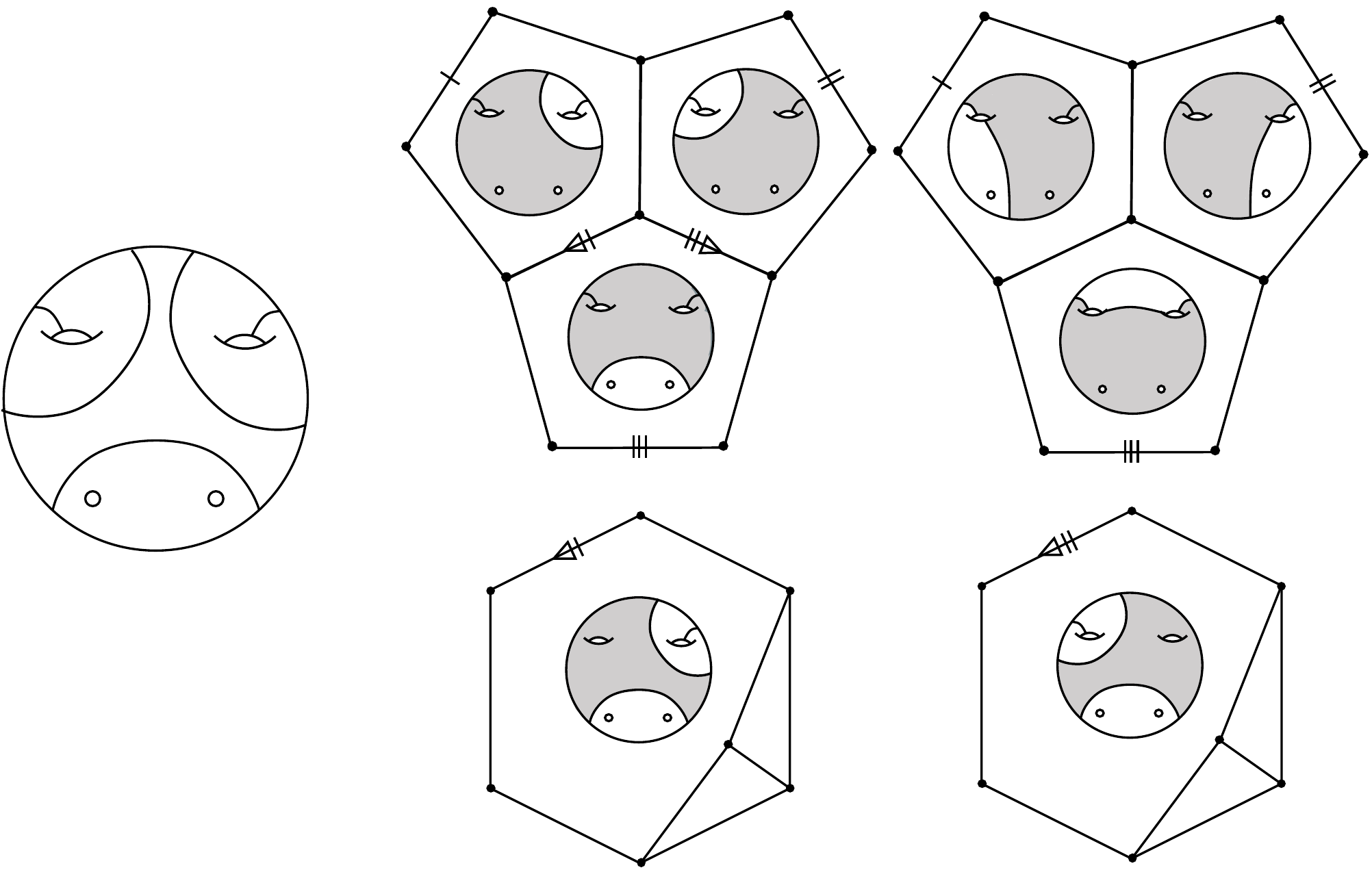} 
\caption{{\bf Left:} $S_{2,2}$ and a pants decomposition $P=C\cup B\cup A$. {\bf Right:} The subgraphs $Z_i \subset X_i \cong X_{1,2}$, for $i =1,2$ and $Z_0 \subset X_0 \cong X_{0,6}$ which is a union of six copies of $Z_{0,5} \subset X_{0,5}$.  These subgraphs are glued to construct $X_{2,2}=X_0\cup X_1\cup X_2$.  Edges are identified as indicated in the figure; note the edges $e_1$ and $e_2$ and the vertex $P$.}
\label{F:X22}
\end{center}
\vspace{.4cm}
\begin{picture}(0,0)(0,0)
\put(10,280){$\small \gamma_1$}
\put(105,280){$\small \gamma_2$}
\put(50,300){$\small \beta_1$}
\put(73,300){$\small \beta_2$}
\put(60,225){$\small \alpha$}
\put(215,295){$\small P$}
\put(290,270){$\small Z_0$}
\put(275,150){$\small Z_1$}
\put(310,150){$\small Z_2$}
\put(190,205){$\small e_1$}
\put(350,205){$\small e_2$}
\put(180,295){$\small e_1$}
\put(250,295){$\small e_2$}
\end{picture}
\end{figure}


\begin{thebibliography}{1}
\bibitem{Aramayona} J. Aramayona, Simplicial embeddings between pants graphs, Geom. Dedicata. \textbf{144} (2010), no. 1, 115-128, MR2580421.
\bibitem{AL} J. Aramayona and C. J. Leininger, Finite rigid sets in curve complexes, J. Topol. Anal. \textbf{5} (2013), no. 2, 183-203, MR3062946.
\bibitem{AL2} J. Aramayona and C. J. Leininger, Exhausting curve complexes by finite rigid sets, Pacific J. Math. \textbf{282} (2016), no. 2, 257-283.
\bibitem{APS} J. Aramayona, H. Parlier, and K. J. Shackleton, Totally geodesic subgraphs of the pants complex, Math. Res. Lett. \textbf{15} (2008), no. 2, 309-320, MR2385643.
\bibitem{HT} A. Hatcher and W. Thurston, A presentation for the mapping class group of a closed orientable surface, Topology. \textbf{19} (1980), no. 3, 221-237, MR579573.
\bibitem{Hern} J. Hern{\'a}ndez Hern{\'a}ndez, Exhaustion of the curve graph via rigid expansions, Glasg. Math. J. 61 (2019), no. 1, 195--230, MR3882310.
\bibitem{Irm} E. Irmak,  Exhausting curve complexes by finite rigid sets on nonorientable surfaces, Preprint, arXiv:1906.09913.
\bibitem{IK} S. Ilbira and M. Korkmaz, Finite Rigid Sets in Curve Complexes of Non-Orientable Surfaces, Preprint, arXiv:1810.07964
\bibitem{Ivanov} N. V. Ivanov, Automorphisms of complexes of curves and of Teichm\"uller spaces, Internat. Math. Res. Notices. \textbf{1997} (1997), no. 14, 651-666, MR1460387.
\bibitem{Korkmaz} M. Korkmaz, Automorphisms of complexes of curves on punctured spheres and on punctured tori, Topology Appl. \textbf{95} (1999), no. 2, 85-111, MR1696431.
\bibitem{Luo} F. Luo, Automorphisms of the complex of curves, Topology. \textbf{39} (2000), no. 2, 283-298, MR1722024.
\bibitem{Mar} D. Margalit, Automorphisms of the pants complex, Duke Math. J. \textbf{121} (2004), no. 3, 457-479, MR2040283.
\bibitem{Rasimate2} R. Maungchang, Exhausting pants graphs of punctured spheres by finite rigid sets, J. Knot Theor. Ramif. \textbf{26} (2017), no. 14, 1750105, MR.
\bibitem{Rasimate1} R. Maungchang, Finite rigid subgraphs of the pants graphs of punctured spheres, Topol. Appl. \textbf{237} (2018), no. C, 37-52, MR.
\bibitem{Minsky} Y. N. Minsky, A geometric approach to the complex of curves on a surface, in: S. Kojima et. al. (Eds.), Topology and {T}eichm\"uller spaces ({K}atinkulta, 1995), World Scientific, River Edge, New Jersey, 1996, pp. 149-158, MR1659683.
\bibitem{SK} K. J. Shackleton, Combinatorial rigidity in curve complexes and mapping class groups, Pacific J. Math. \textbf{230} (2007), no. 1, 217-232, MR2318453.
  
\end{thebibliography}
\end{document}